\documentclass[12pt]{article}

\usepackage{mathtools}
\usepackage{booktabs}
\usepackage[english]{babel} 
\usepackage{amsmath,bm}
\usepackage{bbm}
\usepackage{ amssymb }
\usepackage{graphicx}
\usepackage{array}
\usepackage{multirow}
\usepackage{hhline}
\usepackage[titletoc]{appendix}
\usepackage{subfigure}
\usepackage[normalem]{ulem}
\usepackage[shortlabels]{enumitem}

\usepackage{color}
\usepackage[CJKbookmarks=true,
            bookmarksnumbered=true,
			bookmarksopen=true,
            colorlinks=true,
			citecolor=blue,
			linkcolor=blue,
			anchorcolor=red,
			urlcolor=blue]{hyperref}

\usepackage{eqparbox}

\usepackage{footnote}

\usepackage{smileXX}

\newcommand\eqd{\stackrel{\mathclap{\mbox{d}}}{=}}

\makeatletter
\renewcommand*{\thanks}[1]{%
  \footnotemark
  \protected@xdef\@thanks{\@thanks
    \protect\footnotetext[\arabic{footnote}]{#1}}%
}
\makeatother

\usepackage{setspace}

\usepackage{algpseudocode}

\usepackage{eqparbox}

\usepackage{tabularx}
\usepackage[font = small,labelfont=bf,textfont=it]{caption} 
\usepackage{footnote}
\usepackage{algorithm}

\usepackage[dvipsnames]{xcolor}

\makeatletter
\renewcommand*{\thanks}[1]{%
  \footnotemark
  \protected@xdef\@thanks{\@thanks
    \protect\footnotetext[\arabic{footnote}]{#1}}%
}
\makeatother

\usepackage[utf8]{inputenc}
\usepackage{mathtools}                                  
\usepackage{booktabs}
\usepackage[english]{babel} 
\usepackage{amsmath,amsthm}
\usepackage{ amssymb }
\usepackage{graphicx}
\usepackage{array}
\usepackage{multirow}
\usepackage{hhline}
\usepackage[titletoc]{appendix}

\usepackage[margin=1in,hmarginratio=1:1,top=20mm,columnsep=20pt]{geometry} 
\usepackage{booktabs} 
\usepackage{natbib}
\usepackage{setspace}

\usepackage{graphicx}
\usepackage{subfigure}

\usepackage{algpseudocode}

\usepackage{mathtools}
\usepackage{booktabs}
\usepackage{amsmath,bm}
\usepackage{ amssymb,amsthm }
\usepackage{array}
\usepackage{multirow}
\usepackage{hhline}
\usepackage[titletoc]{appendix}
\usepackage{bbding}
\usepackage{enumitem}

\usepackage{footnote}
\newtheorem{lemma}{Lemma}[section]

\newtheorem{corollary}{Corollary}[section]
\newtheorem{remark}{Remark}[section]
\newtheorem{theorem}{Theorem}[section]

\newtheorem{condition}{Condition}[section]

\newcommand{\AlgUpper}{{\textsc{UCL}}}

\title{Uncertainty Quantification for Bayesian Optimization}

\author{ Rui Tuo\thanks{Two authors contributed equally to this work.} \\ Wm Michael Barnes '64\\ Department of Industrial and Systems Engineering\\ Texas A\&M University \and 
Wenjia Wang \\
Data Science and Analytics Thrust\\
The Hong Kong University of Science\\ and Technology (Guangzhou)\\
and Department of Mathematics\\ 
The Hong Kong University of Science and Technology} 

\begin{document}

\maketitle

\begin{abstract}
Bayesian optimization is a class of global optimization techniques. In Bayesian optimization, the underlying objective function is modeled as a realization of a Gaussian process. Although the Gaussian process assumption implies a random distribution of the Bayesian optimization outputs, quantification of this uncertainty is rarely studied in the literature. In this work, we propose a novel approach to assess the output uncertainty of Bayesian optimization algorithms, which proceeds by constructing confidence regions of the maximum point (or value) of the objective function. These regions can be computed efficiently, and their confidence levels are guaranteed by the uniform error bounds for sequential Gaussian process regression newly developed in the present work. Our theory provides a unified uncertainty quantification framework for all existing sequential sampling policies and stopping criteria.
\end{abstract}

\section{INTRODUCTION}

The empirical and data-driven nature of data science field makes uncertainty quantification one of the central questions that need to be addressed in order to guide and safeguard decision makings. In this work, we focus on Bayesian optimization, which is effective in solving global optimization problems for complex blackbox functions. Our objective is to quantify the uncertainty of Bayesian optimization outputs. Such uncertainty comes from the Gaussian process prior, random input and stopping time. Closed-form solution of the output uncertainty is usually intractable because of the complicated sampling scheme and stopping criterion.

Let $f$ be an underlying 
continuous function over $\Omega$, a compact subset of $\RR^p$. The goal of global optimization is to find the maximum of $f$, denoted by $\max_{x\in \Omega} f(x)$, or the point $x_{max}$ which satisfies $f(x_{max})=\max_{x\in \Omega} f(x)$. In many scenarios, objective functions can be expensive to evaluate. For example, $f$ defined by a complex computer model may take a long time to run. Bayesian optimization is a powerful technique to deal with this type of problems, and has been widely used in areas including designing engineering systems \citep{forrester2008engineering,jones1998efficient}, materials and drug design \citep{frazier2016bayesian,negoescu2011knowledge,solomou2018multi}, chemistry \citep{hase2018phoenics}, deep neural networks \citep{diaz2017effective,klein2016fast}, and reinforcement learning \citep{marco2017virtual,wilson2014using}. 

In Bayesian optimization, $f$ is treated as a realization of a stochastic process, denoted by $Z$. Usually, people assume that $Z$ is a Gaussian process. Every Bayesian optimization algorithm defines a sequential sampling procedure, which successively generates new input points, based on the acquired function evaluations over all previous input points. Usually, the next input point is determined by maximizing an acquisition function. Examples of acquisition functions include probability of improvement \citep{kushner1964new}, expected improvement \citep{huang2006global,jones1998efficient,mockus1978application,picheny2013quantile},  Gaussian process upper confidence bound \citep{azimi2010batch,contal2013parallel,desautels2014parallelizing,srinivas2009gaussian}, entropy search \citep{hennig2012entropy}, predictive entropy search \citep{hernandez2014predictive}, entropy search portfolio \citep{shahriari2014entropy}, knowledge gradient \citep{scott2011correlated,wu2016parallel,wu2017bayesian}, etc. We refer to \cite{frazier2018tutorial,shahriari2016taking} for an introduction to popular Bayesian optimization methods.

Although Bayesian optimization has received considerable attention and numerous techniques have emerged in recent years, how to quantify the uncertainty of the outputs from a Bayesian optimization algorithm is rarely discussed in the literature. Since we assume that $f$ is a \textit{random realization} of $Z$, $x_{max}$ and $f(x_{max})$ should also be random. However, the highly nontrivial distributions of $x_{max}$ and $f(x_{max})$ make uncertainty quantification rather challenging. 

In this work, we develop efficient methods to construct confidence regions of $x_{max}$ and $f(x_{max})$ for Bayesian optimization algorithms, where function $f$ is a realization of Gaussian process $Z$. Our uncertainty quantification method \textit{does not} rely on the specific formulae or strategies, and can be applied to all existing methods in an abstract sense. We show that by using the collected data of any instance algorithm of Bayesian optimization, Algorithm \ref{alg:upper} gives a confidence upper limit with theoretical guarantees of their confidence level in Corollary \ref{coro:sequential}. To the best of our knowledge, this is the \textit{first} theoretical result of the uncertainty quantification on the maximum estimator for Bayesian optimization, under the assumption that $f$ is a realization of a Gaussian process. Compared with the traditional point-wise predictive standard deviation of Gaussian process regression, denoted by $\sigma(x)$, our bound is only inflated by a factor proportional to $\sqrt{\log (e \sigma/\sigma(x))}$, where $\sigma$ is the prior standard deviation.

It is worth noting that uncertainty quantification typically differs from convergence analysis of algorithms. In Bayesian optimization, the latter topic has been studied more often. See, for instance, \cite{bect2016supermartingale,calvin1997average,calvin2005one,ryzhov2016convergence,vazquez2010convergence,yarotsky2013univariate}. These analyses provide theoretical guarantee on the convergence of Bayesian optimization algorithms, but do not directly lead to techniques for uncertainty quantification. Recall that in this work, we assume that the underlying function $f$ is a realization of a Gaussian process, and therefore, the sample path properties of $f$, such as the smoothness, should be governed by the covariance function of the Gaussian process. This Gaussian process assumption differs from those in some existing works in the analysis of Bayesian optimization, e.g., \cite{bull2011convergence,astudillo2019bayesian,yarotsky2013univariate}, where the underlying function $f$ is assumed to be a \textit{deterministic} function satisfying pre-specified smoothness conditions.

The rest of this paper is structured as follows. In Section \ref{sec:pre}, we present some preliminaries, including an introduction to Gaussian process regression and Bayesian optimization. Section \ref{sec:fixeddesign} presents uncertainty quantification results under fixed designs. Section \ref{Sec:UQ} introduces our methods and main theoretical results. How to calibrate the constant in our method is introduced in Section \ref{Sec:K}. Numerical results are presented in Section \ref{Sec:numexp}. Conclusions and discussion are made in Section \ref{sec:dis}. Technical details are given in the Appendix.

\section{PRELIMINARIES}\label{sec:pre}
In this section, we provide a brief introduction to Gaussian process regression and review some existing methods in Bayesian optimization.

\subsection{Gaussian Process Regression} 
Recall that in Bayesian optimization, the objective function $f$ is assumed to be a realization of a Gaussian process $Z$. In this work, we suppose that $Z$ is stationary and has mean zero, variance $\sigma^2$ and correlation function $\Psi$, i.e., $\text{Cov}(Z(x),Z(x'))=\sigma^2\Psi(x-x')$ with $\Psi(0)=1$. Under certain regularity conditions, Bochner's theorem \citep{wendland2004scattered} suggests that the Fourier transform (with a specific choice of the constant factor) of $\Psi$, denoted by $\tilde{\Psi}$, is a probability density function and satisfies the inversion formula $\Psi(x)=\int_{\mathbb{R}^p}\cos(\omega^T x)\tilde{\Psi}(\omega) d \omega$. We call $\tilde{\Psi}$ the \textit{spectral density} of $\Psi$. Some popular choices of correlation functions and their spectral densities are discussed in Section \ref{Sec:A0}. Throughout this work, we further assume $\Psi$ satisfies the following condition. For a vector $\omega=(\omega_1,\ldots,\omega_p)^T$, define its $l_1$-norm as $\|\omega\|_1=|\omega_1|+\ldots+|\omega_p|$.
\begin{condition}\label{C1}
The correlation function $\Psi$ has a spectral density, denoted by $\tilde{\Psi}$, and
\begin{eqnarray}\label{A0}
    A_0=\int_{\mathbb{R}^p} \|\omega\|_1\tilde{\Psi}(\omega) d \omega<+\infty. 
\end{eqnarray}
\end{condition}

\begin{remark}
The $l_1$-norm in \eqref{A0} can be replaced by the usual Euclidean norm. However, we use the former here because they usually have explicit expressions. See Section \ref{Sec:A0} for details.
\end{remark}

\begin{remark}
Condition \ref{C1} imposes a smoothness condition on the correlation function $\Psi$, which is equivalent to the mean squared differentiability \citep{stein2012interpolation} of the Gaussian process $Z$. Note that the mean squared differentiability differs from the sample path differentiability. We refer to \cite{driscoll1973reproducing,steinwart2019convergence} for results on the relationship between the sample path smoothness of $Z$ (thus $f$) and the smoothness of correlation function $\Psi$. 
\end{remark}

Suppose the set of points $X=(x_1,\ldots,x_n)$ is given. Then $f$ can be reconstructed via Gaussian process regression. Let $Y=(Z(x_1),\ldots,Z(x_n))^T$ be the vector of evaluations of the Gaussian process at points $x_1,...,x_n$. The following results are well-known and can be found in \cite{rasmussen2006gaussian}. For any untried point $x$, conditional on $Y$, $Z(x)$ follows a normal distribution. The conditional mean and variance of $Z(x)$ are
\begin{align}
\mu(x) :=& \mathbb{E}[Z(x)|Y]=r^T(x) K^{-1} Y,\label{mean}\\
\sigma^2(x) :=& \text{Var}[Z(x)|Y]=\sigma^2(1-r^T(x) K^{-1} r(x)),\label{variance}
\end{align}
where $r(x)=(\Psi(x-x_1),\ldots,\Psi(x-x_n))^T, K=(\Psi(x_j-x_k))_{j k}$. Since we assume that $f$ is a realization of $Z$, $\mu(x)$ can serve as a reconstruction of $f$. 

\subsection{Bayesian Optimization}\label{Sec:existing}
In Bayesian optimization, we evaluate $f$ over a set of input points, denoted by $x_1,\ldots,x_n$. We call them the \textit{design points}, because these points can be chosen according to our will. There are two categories of strategies to choose design points. We can choose all the design points before we evaluate $f$ at any of them. Such a design set is call a \textit{fixed design}. An alternative strategy is called \textit{sequential sampling}, in which the design points are not fully determined at the beginning. Instead, points are added sequentially, guided by the information from the previous input points and the corresponding acquired function values. An instance algorithm defines a sequential sampling scheme which determines the next input point $x_{n+1}$ by maximizing an \textit{acquisition function} $a(x;X_n,Y_n)$, where $X_n=(x_1,\ldots,x_n)$ consists of previous input points, and $Y_n =(f(x_1),\ldots,f(x_n))^T$ consists of corresponding outputs. The acquisition function can be either deterministic or random given $X_n,Y_n$. A general Bayesian optimization procedure under sequential sampling scheme is shown in Algorithm \ref{alg:bayesopt}.

\begin{algorithm}
\caption{Bayesian optimization (described in \cite{shahriari2016taking})}
\begin{algorithmic}[1]
\State \textbf{Input:} A Gaussian process prior of $f$, initial observation data $X_1,Y_1$. 
\State \textbf{for} $n=1,2\ldots,$ \textbf{do}\\
Find $x_{n+1} = \argmax_{x\in \Omega}a(x;X_n,Y_n)$, evaluate $f(x_{n+1})$, update data and the posterior probability distribution on $f$.
\State \textbf{Output: } The point evaluated with the largest $f(x)$.
\end{algorithmic}\label{alg:bayesopt}
\end{algorithm}

A number of acquisition functions are proposed in the literature, for example: 
\begin{enumerate}[leftmargin=*]
    \item Expected improvement (EI) \citep{jones1998efficient,mockus1978application}, with the acquisition function $a_{\text{EI}}(x;X_n,Y_n) := \EE((Z(x) - y_n^*) \mathbf{1}(Z(x) - y_n^*)|X_n, Y_n),$ where $\mathbf{1}(\cdot)$ is the indicator function, and $y_n^* = \max_{1\leq i\leq n} f(x_i)$.
    \item Gaussian process upper confidence bound \citep{srinivas2009gaussian}, with the acquisition function $a_{\text{UCB}}(x;X_n,Y_n) := \mu_n(x) + \beta_n \sigma_n(x)$, where $\beta_n$ is a parameter, and $\mu_n(x)$ and $\sigma_n(x)$ are posterior mean and variance of $f$ after $n$th iteration, respectively.
    \item Predictive entropy search \citep{hernandez2014predictive}, with the acquisition function $a_{\text{PES}}(x;X_n,Y_n) := \mathcal{H}(y|X_n,Y_n,x)-\EE_{p(x_{max}|X_n,Y_n)}\mathcal{H}(y|X_n,Y_n,x,x_{max})$, where $\mathcal{H}(y|X_n,Y_n,x)$ and $\mathcal{H}(y|X_n,Y_n,x,x_{max})$ are the differential entropy of the posterior distribution $p(y|X_n,Y_n,x)$ and $p(y|X_n,Y_n,x,x_{max})$, respectively, and the expectation can be approximated via Thompson samples.
    Another entropy search acquisition function is introduced by \cite{hennig2012entropy}, who also provide an efficient way to approximate the distribution of $x_{max}$ based on the Gaussian process prior.
\end{enumerate}
Among the above acquisition functions, $a_{\text{EI}}$ and $a_{\text{UCB}}$ are deterministic functions of $(x,X_n,Y_n)$, whereas $a_{\text{PES}}$ is random in practice because Thompson sampling depends on a random sample from the posterior Gaussian process. We refer to \cite{shahriari2016taking} for general discussions and popular methods in Bayesian optimization.

In practice, one also needs to determine when to stop Algorithm \ref{alg:bayesopt}. Usually, decisions are made in consideration of the budget and the accuracy requirement. For instance, practitioners can stop Algorithm \ref{alg:bayesopt} after finishing a fixed number of iterations \citep{frazier2018tutorial} or no further significant improvement of function values can be made \citep{acerbi2017practical}. Although stopping criteria plays no role in the analysis of the algorithms' asymptotic behaviors, they can greatly affect the output uncertainty.

\section{OPTIMIZATION WITH GAUSSIAN PROCESS REGRESSION UNDER FIXED DESIGNS}\label{sec:fixeddesign}


Before investigating the more important sequential sampling schemes, we shall first consider fixed designs in this section, because the latter situation is simpler and will serve as an important intermediate step to the general problem in Section \ref{Sec:UQ}.

\subsection{Uniform Error Bound}


Although the conditional distribution of $Z(x)$ is simple as shown in (\ref{mean})-(\ref{variance}), those for $x_{max}$ and $Z(x_{max})$ are highly non-trivial because they are nonlinear functionals of $Z$. In this work, we construct confidence regions for the maximum points and values using a uniform error bound for Gaussian process regression, given in Theorem \ref{thm:undersmooth}. 
We will use the notion $a\vee b:=\max(a,b)$. Also, we shall use the convention $0/0=0$ in all statements in this article related to error bounds.


\begin{theorem}\label{thm:undersmooth}
Suppose Condition \ref{C1} holds. Let $M=\sup_{x\in \Omega}\frac{Z(x) - \mu(x)}{\sigma(x)\log^{1/2} (e \sigma/\sigma(x))}$, where $\mu(x)$ and $\sigma(x)$ are given in (\ref{mean}) and (\ref{variance}), respectively.  Then the followings are true.
\begin{enumerate}[leftmargin=*]
\item $\mathbb{E}M\leq C_0 \sqrt{p(1 \vee \log(A_0 D_\Omega))}$, where $C_0$ is a universal constant, $A_0$ is as in Condition \ref{C1}, and $D_\Omega = \text{diam}(\Omega)$ is the Euclidean diameter of $\Omega$. 
\item For any $t > 0$, $\mathbb{P}(M-\mathbb{E}M>t)\leq e^{-t^2/2}.$
\end{enumerate}
\end{theorem}

In practice, Part 2 of Theorem \ref{thm:undersmooth} is hard to use directly because $\mathbb{E}M$ is difficult to calculate accurately. Instead, we can replace $\mathbb{E}M$ by its upper bound in Part 1 of Theorem \ref{thm:undersmooth}. We state such a result in Corollary \ref{coro:practice}. Its proof is trivial.

\begin{corollary}\label{coro:practice}
Under the conditions and notation of Theorem \ref{thm:undersmooth}, for any constant $C$ such that $\mathbb{E}M\leq C \sqrt{p(1 \vee \log(A_0 D_\Omega))}$, we have that for any $t>0$,
$$\mathbb{P}(M-C \sqrt{p(1 \vee \log(A_0 D_\Omega))}>t)\leq e^{-t^2/2}. $$
\end{corollary}

To use Corollary \ref{coro:practice}, we need to determine the universal constant $C$ and the moment of the spectral density $A_0$. According to our numerical simulations in Section \ref{Sec:K} and Section \ref{App:simu} of the Supplementary material, we recommend using $C=1$ in practice. We shall discuss the calculation of $A_0$ in Section \ref{Sec:A0}. 



\subsection{Uncertainty Quantification}

In light of Corollary \ref{coro:practice}, we can construct a confidence upper limit of $f$. Algorithm \ref{alg:upperfix} describes how to compute the confidence upper limit of $f$ at a given untried $x$. For notational simplicity, we regard the dimension $p$, the variance $\sigma^2$, the moment $A_0$ and the universal constant $C$ as global variables so that Algorithm \ref{alg:upperfix} has access to all of them. 

\begin{algorithm}
\caption{Uniform confidence upper limit at a given point: \AlgUpper $(x,t,X,Y)$}
\begin{algorithmic}[1]
\State \textbf{Input}: Untried point $x$, significance parameter $t$, data $X=(x_1,\ldots,x_n)^T, Y$.
\State Set $r=(\Psi(x-x_1),\ldots,\Psi(x-x_n))^T, K=(\Psi(x_j-x_k))_{j k}$. Calculate
$
\mu=r^T K^{-1} Y,
s =\sqrt{\sigma^2(1-r^T K^{-1} r)}.
$
\State \textbf{Output}: $\mu+s\sqrt{\log (e\sigma/s)}(C \sqrt{p(1 \vee \log(A_0 D_\Omega))}+t).$
\end{algorithmic}\label{alg:upperfix}
\end{algorithm}


Based on the \AlgUpper ~function in Algorithm \ref{alg:upper}, we can construct a confidence region for $x_{max}$ and a confidence interval for $f(x_{max})$. These regions do not have explicit expressions. However, they can be approximated by calling \AlgUpper ~with many different $x$'s. Let $Y=(f(x_1),\ldots,f(x_n))^T$.
The confidence region for $x_{max}$ is defined as
\begin{eqnarray}\label{CR}
    CR_t:=\left\{x\in\Omega:\AlgUpper(x,t,X,Y)\geq \max_{1\leq i\leq n}f(x_i)\right\}.
\end{eqnarray}
The confidence interval for $f(x_{max})$ is defined as 
\begin{eqnarray}\label{CI}
CI_t:=\left[\max_{1\leq i\leq n}f(x_i), \max_{x\in\Omega}\AlgUpper(x,t,X,Y)\right].
\end{eqnarray}


It is worth noting that the probability in Corollary \ref{coro:practice} is \textit{not} a posterior probability. Therefore, the regions given by (\ref{CR}) and (\ref{CI}) should be regarded as frequentist confidence regions under the Gaussian process model, rather than Bayesian credible regions.
Such a frequentist nature has an alternative interpretation, shown in
Corollary \ref{coro:1}. Corollary \ref{coro:1} simply translates Corollary \ref{coro:practice} from the language of stochastic processes to a deterministic function approximation setting, which fits the Bayesian optimization framework better. It shows that $CR_t$ in (\ref{CR}) and $CI_t$ in (\ref{CI}) are confidence region of $x_{max}$ and $f(x_{max})$ with confidence level $1-e^{-t^2/2}$, respectively. In particular, to obtain a $95\%$ confidence region, we use $t=2.448$.

\begin{corollary}\label{coro:1}
Let $C(\Omega)$ be the space of continuous functions on $\Omega$ and $\mathbb{P}_Z$ be the law of $Z$. Then there exists a set $B\subset C(\Omega)$ so that $\mathbb{P}_Z(B)\geq 1-e^{-t^2/2}$ and for any $f\in B$, its maximum point $x_{max}$ is contained in $CR_t$ defined in (\ref{CR}), and $f(x_{max})$ is contained in $CI_t$ defined in (\ref{CI}).
\end{corollary}

In practice, the shape of $CR_t$ can be highly irregular and representing the region of $CR_t$ can be challenging. If $\Omega$ is of one or two dimensions, we can choose a fine mesh over $\Omega$ and call \AlgUpper$(x,t,X,Y)$ for each mesh grid point $x$. In a general situation, we suggest calling \AlgUpper$(x,t,X,Y)$ with randomly chosen $x$'s and using the $k$-nearest neighbors algorithm to represent $CR_t$.

\subsection{Calculating $A_0$}\label{Sec:A0}

For an arbitrary $\Psi$, calculation of $A_0$ in (\ref{A0}) can be challenging. Fortunately, for two most popular correlation functions in one dimension, namely the Gaussian and the Mat\'ern correlation functions \citep{rasmussen2006gaussian}, $A_0$ can be calculated in closed form. The results are summarized in Table \ref{tab:kernels}. In Table \ref{tab:kernels}, $\Gamma(\cdot)$ is the Gamma function and $K_\nu(\cdot)$ is the modified Bessel function of the second kind.

\begin{table*}[h]
    \centering
\begin{tabular}{c c c}
\hline
  Correlation family  & Gaussian & Mat\'ern\\
    \hline
    Correlation function & $\exp\{-(x/\theta)^2\}$ & $\frac{1}{\Gamma(\nu)2^{\nu-1}}\left(\frac{2\sqrt{\nu}|x|}{\theta}\right)^\nu K_\nu\left(\frac{2\sqrt{\nu}|x|}{\theta}\right)$\\
    Spectral density & $\frac{\theta}{2\sqrt{\pi}}\exp\{-\omega^2\theta^2/4\}$ & $\frac{\Gamma(\nu+1/2)}{\Gamma(\nu)\sqrt{\pi}}\left(\frac{4\nu}{\theta^2}\right)^\nu\left(\omega^2+\frac{4\nu}{\theta^2}\right)^{-(\nu+1/2)}$\\
    $A_0$ & $\frac{2}{\sqrt{\pi}\theta}$ & $\frac{4\sqrt{\nu}\Gamma(\nu+1/2)}{\sqrt{\pi}(2\nu-1)\theta\Gamma(\nu)}$ for $\nu>1/2$
\\
    \hline
\end{tabular}\caption{Gaussian And Mat\'ern Correlation Families.}\label{tab:kernels}
\end{table*}

For multi-dimensional problems, a common practice is to use \textit{product} correlation functions. Specifically, suppose $\Psi_1,\ldots,\Psi_p$ are one-dimensional correlation functions. Then their product $\Psi(x)=\prod_{i=1}^{p}\Psi(x_i)$ forms a $p$-dimensional correlation function, where $x=(x_1,\ldots,x_p)^T$. If a product correlation function is used, the calculation of $A_0$ is easy. It follows from the elementary properties of Fourier transform that $\tilde{\Psi}(x)=\prod_{i=1}^p\tilde{\Psi}_i(x_i)$. Let $X_i$ be a random variable with probability density function $\Psi_i$. Then $A_0=\sum_{i=1}^p \mathbb{E}|X_i|$, i.e., the value of $A_0$ corresponding to a product correlation function is the sum of those given by the marginal correlation functions. If each $\Psi_i$ is either a Gaussian or Mat\'ern correlation function, then $\mathbb{E}|X_i|$'s can be read from Table \ref{tab:kernels}.

\section{UNCERTAINTY QUANTIFICATION FOR SEQUENTIAL SAMPLINGS}\label{Sec:UQ}

In Bayesian optimization, sequential samplings are more popular, because such approaches can utilize the information from the previous responses and choose new design points in the area which is more likely to contain the maximum points. 
In this section, we present our uncertainty quantification methodology for sequential samplings, as well as the theoretical guarantees.


\subsection{Methodology}\label{subsec:method}
In this subsection, we construct confidence regions for the maximum points and values under sequential sampling scheme, as presented in Algorithm \ref{alg:upper}. In the rest of this work, let $T$ be the number of iterations when an instance of Algorithm \ref{alg:bayesopt} stops and $D_\Omega$ be the diameter of $\Omega$. It is worth noting that under sequential sampling scheme, $T$ can be a random variable, which introduces additional randomness of the confidence interval, and complicates the theoretical analysis. Given $n$, we denote 
$X_{1:n}=(x_1,\ldots,x_{m_n}),$
where each $x_i$ is corresponding to one data point and $m_n$ is the number of sampled points after $n$ iterations of the algorithm, and $Y_{1:n}=(f(x_1),\ldots,f(x_{m_n}))^T$. In this work, we allow $m_n\geq 1$, which means that we can sample one point or a batch of points at a time in each iteration. Let $g(x)=\AlgUpper(x,t,X_{1:T},Y_{1:T})$,
\begin{align}
    CR_t^{\mathbf{seq}}:=&\left\{x\in\Omega:g(x) \geq \max_{1\leq i\leq m_T}f(x_i)\right\},\label{seqR}\\
CI_t^{\mathbf{seq}}:=&\left[\max_{1\leq i\leq m_T}f(x_i), \max_{x\in\Omega}g(x)\right].\label{seqI}
\end{align}

\begin{algorithm}
\caption{Confidence regions for $x_{\max}$ and $f(x_{\max})$}
\begin{algorithmic}[1]
\State \textbf{Input:} Significance parameter $t$, data $X_{1:T}, Y_{1:T}$ collected from an instance of Bayesian optimization algorithm.
\State For point $x\in \Omega$, set $r(x)=(\Psi(x-x_1),\ldots,\Psi(x-x_{m_T}))^T, K=(\Psi(x_j-x_k))_{j k}$. Calculate
    \begin{eqnarray}
\mu_T(x)&=&r(x)^T K^{-1} Y_{1:T},\label{seqmeanT}\\
s_T(x) &=&\sqrt{\sigma^2(1-r(x)^T K^{-1} r(x))}.\label{seqvarianceT}
\end{eqnarray}\\
Compute ${\rm \AlgUpper}(x,t,X_{1:T},Y_{1:T})$ via Algorithm \ref{alg:upperfix}.
\State Let $g(x)=\AlgUpper(x,t,X_{1:T},Y_{1:T})$. Calculate
    $CR_t^{\mathbf{seq}}$ and $CI_t^{\mathbf{seq}}$ via \eqref{seqR} and \eqref{seqI}, respectively.

\State \textbf{Output: } The confidence region $CR_t^{\mathbf{seq}}$ for $x_{max}$ and the confidence interval $CI_t^{\mathbf{seq}}$ for $f(x_{max})$.
\end{algorithmic}\label{alg:upper}
\end{algorithm}

In Section \ref{sec:theory}, we will show that under the condition that $f$ is a realization of $Z$, $CR_t^{\mathbf{seq}}$ and $CI_t^{\mathbf{seq}}$ are confidence regions of $x_{max}$ and $f(x_{max})$, respectively, with a simultaneous confidence level at least $1-e^{-t^2/2}$, respectively. In particular, to obtain a $95\%$ confidence region, we use $t=2.448$. The calculation of $A_0$ follows the discussion in Section \ref{Sec:A0}, and we recommend using $C=1$ in practice, as in Algorithm \ref{alg:upperfix}.

\subsection{Theory}\label{sec:theory}

To facilitate our mathematical analysis, we first state the general Bayesian optimization framework in a rigorous manner. Recall that we assume that $f$ is a realization of a Gaussian process $Z$ with correlation function $\Psi$. From this Bayesian point of view, we shall not differentiate $f$ and $Z$ in this section.

Denote the vectors of input and output points in the $n$th iteration as $X_n$ and $Y_n$, respectively. Let $X_{1:n}$ and $Y_{1:n}$ be as in Section \ref{subsec:method}. Because $X_{1:n}$ and $Y_{1:n}$ are random, the data $(X_{1:n},Y^T_{1:n})$ is associated with the $\sigma$-algebra $\mathcal{F}_n$, defined as the $\sigma$-algebra generated by $(X_{1:n},Y^T_{1:n})$. When the algorithm just starts, the data is an empty set, which is associated with the trivial $\sigma$-algebra $\mathcal{F}_0$.
In each sampling-evaluation iteration, a sequential sampling strategy, which determines the next sample point or a batch of points based on the current data, is applied. 
Clearly, such strategy should not depend on unobserved data. After each sampling-evaluation iteration, a stopping criterion is checked and to determine whether to terminate the algorithm. A stopping decision should depend only on the current data and/or prespecified values such as computational budget, and should not depend on unobserved data either. Let $T$ be the number of iterations when the algorithm stops. Then a Bayesian optimization algorithm must satisfy the following conditions.

\begin{enumerate}[leftmargin=*]
    \item Conditional on $\mathcal{F}_{n-1}$, $X_n$ and $Z$ are mutually independent for $n=1,2,\ldots$.
    \item $T$ is a stopping time with respect to the filtration $\{\mathcal{F}_n\}_{n=0}^\infty$. We further require $1\leq T<+\infty, a.s.,$ to ensure a meaningful Bayesian optimization procedure.
\end{enumerate}

We shall establish a generic theory that bounds the uniform prediction error, which can be applied to any instance algorithms of Bayesian optimization. It is worth noting that several literature, including \cite{sniekers2015credible,yoo2016supremum,yang2017non,kuriki2019optimal,azzimonti2019profile,azais2010simultaneous}, investigate uncertainty quantification methods which are not within the Bayesian optimization or sequential sampling scheme, and cannot be directly applied to quantify the uncertainties of outputs of Bayesian optimization. 

In Bayesian optimization, sequential samplings are more popular, because such approaches can utilize the information from the previous responses and choose new design points in the area which is more likely to contain the maximum points. Similar to Section \ref{sec:fixeddesign}, we first quantify the uncertainty of $Z(\cdot)-\mu_T(\cdot)$. Note that $Z(\cdot)-\mu_T(\cdot)$ is generally \textit{not} a Gaussian process, because in the sequential samplings situation, the stopping time $T$ is random. Nonetheless, an error bound similar to that in Theorem \ref{thm:undersmooth} is still valid. 
In the following theorem, we define
\begin{align}
\mu_n(x) :=& r^T_n(x) K^{-1}_n Y_{1:n},\label{seqmean}\\
\sigma^2_n(x) :=& \sigma^2(1-r^T_n(x) K_n^{-1} r_n(x)),\label{seqvariance}
\end{align}
where $r_n(x)=(\Psi(x-x_1),\ldots,\Psi(x-x_{m_n}))^T, K_n=(\Psi(x_j-x_k))_{j k}$. 

\begin{theorem}\label{thm:sequential}
{\rm \textbf{(Uncertainty quantification for sequential samplings)}} Suppose Condition \ref{C1} holds. Given an instance of Bayesian optimization algorithm, let
$$M_n=\sup_{x\in \Omega}\frac{Z(x) - \mu_n(x)}{\sigma_n(x)\log^{1/2} (e \sigma/\sigma_n(x))},$$ where $\mu_n(x)$ and $\sigma_n(x)$ are given in \eqref{seqmean} and \eqref{seqvariance}, respectively. Then for any $t>0$, 
\begin{eqnarray}\label{boundsequential}
\mathbb{P}(M_T-C \sqrt{p(1 \vee \log(A_0 D_\Omega))}>t)\leq e^{-t^2/2},
\end{eqnarray}
where $C,A_0,D_\Omega$ are the same as in Corollary \ref{coro:practice}.
\end{theorem}
The proof of Theorem \ref{thm:sequential} can be found in Appendix \ref{sec:pfthm2}. The major difficulty of proving Theorem \ref{thm:sequential} is that the stopping time $T$ is random, which introduces extra uncertainties of the output of a Bayesian optimization algorithm. The probability bound \eqref{boundsequential} has a major advantage: the constant $C$ is independent of the specific Bayesian optimization algorithm, and it can be chosen the same as that for fixed designs. This suggests that when calibrating $C$ via numerical simulations (see Section \ref{Sec:K} and Appendix \ref{App:simu}), we only need to simulate for fixed-design problems, and the resulting constant $C$ can be used for the uncertainty quantification of all past and possible future Bayesian optimization algorithms. The dimension $p$ does not influence our uncertainty quantification a lot, in the sense that only $\sqrt{p}$ appears in \eqref{boundsequential}. However, the dimension will strongly influence the performance of Bayesian optimization \citep{bull2011convergence}, which could lead to a large confidence region.

Analogous to Corollary \ref{coro:1}, we can restate Theorem \ref{thm:sequential} under a deterministic setting in terms of Corollary \ref{coro:sequential}. In this situation, we have to restrict ourselves to \textit{deterministic} instances of Bayesian optimization algorithms, in the sense that the sequential sampling strategy is a deterministic map, such as the first two examples in Section \ref{Sec:existing}.

\begin{corollary}\label{coro:sequential}
Let $C(\Omega)$ be the space of continuous functions on $\Omega$ and $\mathbb{P}_Z$ be the law of $Z$. Given a deterministic instance of Bayesian optimization algorithm, there exists a set $B\subset C(\Omega)$ so that $\mathbb{P}_Z(B)\geq 1-e^{-t^2/2}$ and for any $f\in B$, its maximum point $x_{max}$ is contained in $CR_t^{\mathbf{seq}}$ defined in \eqref{seqR}, and $f(x_{max})$ is contained in $CI_t^{\mathbf{seq}}$ defined in \eqref{seqI}.
\end{corollary}

\section{CALIBRATING $C$ VIA SIMULATION STUDIES}\label{Sec:K}
To construct confidence regions \eqref{CR} and \eqref{seqR}, and confidence intervals \eqref{CI} and \eqref{seqI}, we need to specify the constant $C$. An upper bound of the constant $C$ in Theorem \ref{thm:undersmooth} can be obtained by examine the proof of Lemma \ref{thm133} and Theorem \ref{ourthmjasa}. However, this theoretical upper bound can be too large for practical use. In this work, we consider estimating $C$ via numerical simulation. The details are presented in Appendix \ref{App:simu}. Here we outline the main conclusions of our simulation studies.

Our main conclusions are: 1) $C = 1$ is a robust choice for most of the cases; 2) for the cases with Gaussian correlation functions or small $A_0D_\Omega$, choosing $C=1$ may lead to very conservative confidence regions. We suggest practitioners first consider $C=1$ to obtain robust confidence regions. When users believe that this robust confidence region is too conservative, they can use the value in Table \ref{p1tab} or \ref{p2tab} corresponding to their specific setting, or run similar numerical studies as in Appendix \ref{App:simu} to calibrate their own $C$.

\section{NUMERICAL EXPERIMENTS}\label{Sec:numexp}

In this section, we conduct several numerical studies to compare the performance between the proposed confidence interval $CI_t^{\mathbf{seq}}$ as in \eqref{seqI} and the naive bound of Gaussian process. The nominal confidence levels are 95\% for both methods. The naive 95\% confidence upper bound, denoted by $CI_G$, is defined as the usual pointwise upper bound of Gaussian process, i.e., 
\begin{align}\label{eq:naiveci}
    CI_G:=&\left[\max_{1\leq i\leq m_T}f(x_i), \max_{x\in \Omega}\mu_T(x) + q_{0.05}\sigma_T(x) \right],
\end{align}
where $q_{0.05}$ is the 0.95 quantile of the standard normal distribution, $\mu_T(x)$ and $\sigma_T(x)$ are given in \eqref{seqmeanT} and \eqref{seqvarianceT}, respectively. As suggested in Section \ref{Sec:K}, we use $C = 1$ and $t=2.448$ in $CI_t^{\mathbf{seq}}$. 

\subsection{Well-Specified Gaussian Process}\label{subsec:eg1}
We first consider that the underlying truth is a Gaussian process with known covariance function. We consider the Mat\'ern correlation functions (see Table \ref{tab:kernels}) with $\nu = 1.5,2.5,3.5$, and $A_0D_\Omega = 25$. We simulate Gaussian processes on $\Omega =[0,1]^2$ for each $\nu$. We use optimal Latin hypercube designs \citep{stocki2005method} to generate five initial points. We employs $a_{\text{UCB}}$ (see Section \ref{Sec:existing}) as the acquisition function, where the parameter $\beta_n$ is chosen as the theoretically optimal parameter, suggested by \cite{srinivas2009gaussian}.

We repeat the above procedure 100 times to estimate the coverage rate by calculating the relative frequency of the event $f(x_{\max})\in CI_t^{\mathbf{seq}}$ or $f(x_{\max})\in CI_G$. We also compare $CI_t^{\mathbf{seq}}$ and $CI_G$ with the ``optimal upper bound'' in the sense that we choose a constant $a_\nu$ and the confidence upper bound
\begin{align*}
    CI_a:=&\left[\max_{1\leq i\leq m_T}f(x_i), \max_{x\in \Omega}\mu_T(x) + a_\nu\sigma_T(x) \right],
\end{align*}
such that the relative frequency of the event $f(x_{\max})\in CI_a$ is exactly 95\%, where $a_\nu$ only depends on $\nu$. Then we plot the coverage rate of $CI_t^{\mathbf{seq}}$ and $CI_G$, and the width of $CI_t^{\mathbf{seq}}$, $CI_G$, and $CI_a$ under 5, 10, 15, 20, 25, 30 iterations, respectively.

Panels 1 and 2 in Figure \ref{figall} shows the coverage rates and the width of the confidence intervals under different smoothness with $\nu=1.5,2.5,3.5$.
From the Panel 1 in Figure \ref{figall}, we find that the coverage rate of $CI_t^{\mathbf{seq}}$ is almost 100\% for all the experiments, while $CI_G$ has a lower coverage rate no more than $75\%$. Thus the proposed method is conservative while the naive one is permissive. Such a result shows that using the naive method may be risky in practice, because the naive one underestimates the uncertainties. The coverage results support our theory and conclusions made in Section \ref{sec:theory}. As shown by the Panel 2 in Figure \ref{figall}, the widths of $CI_t^{\mathbf{seq}}$ are about five times of $CI_G$, and about 2-2.5 times of $CI_a$. The ratio decreases as the number of iterations increases. The inflation in the width of confidence intervals is the cost of gaining confidence.

\begin{figure*}[h!]
  \centering
  \includegraphics[width=0.45\textwidth]{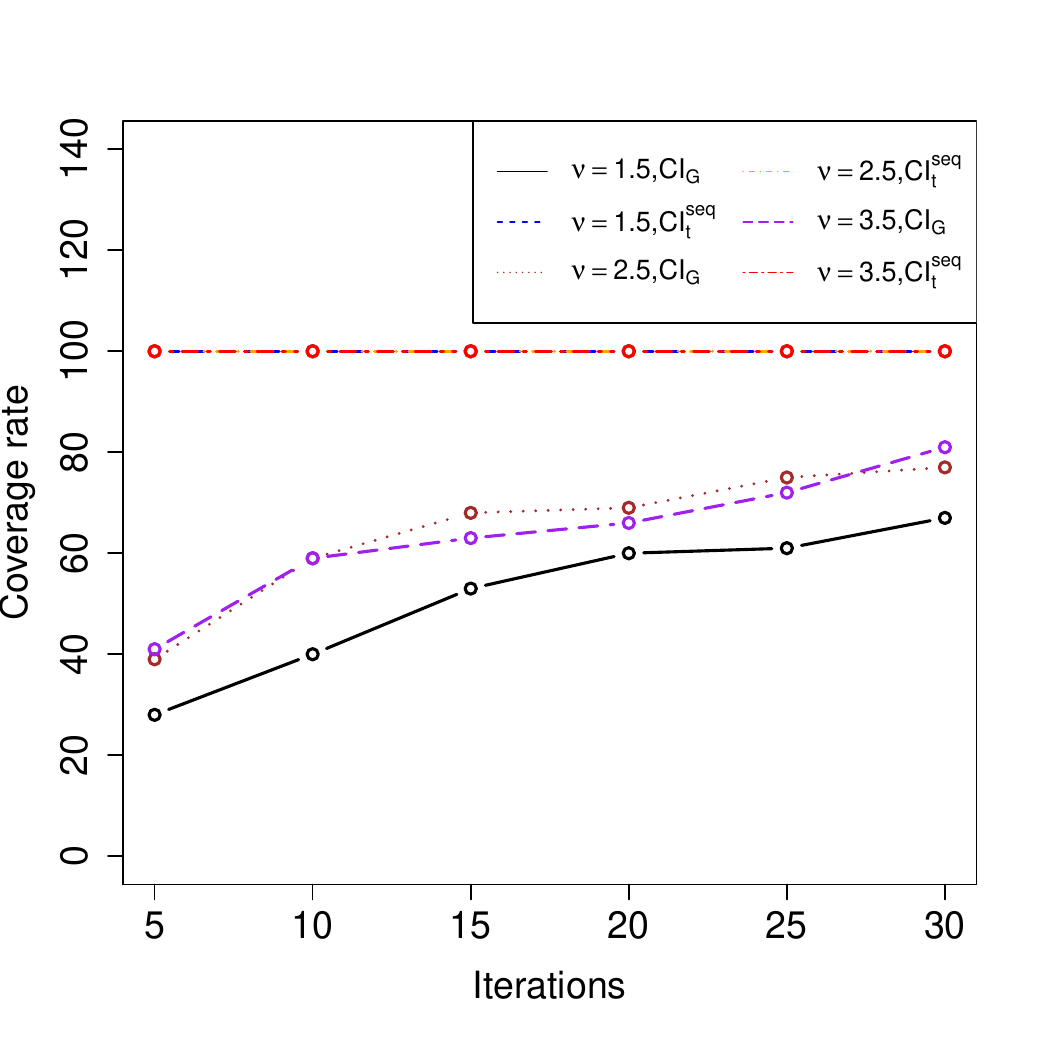}
        \includegraphics[width=0.45\textwidth]{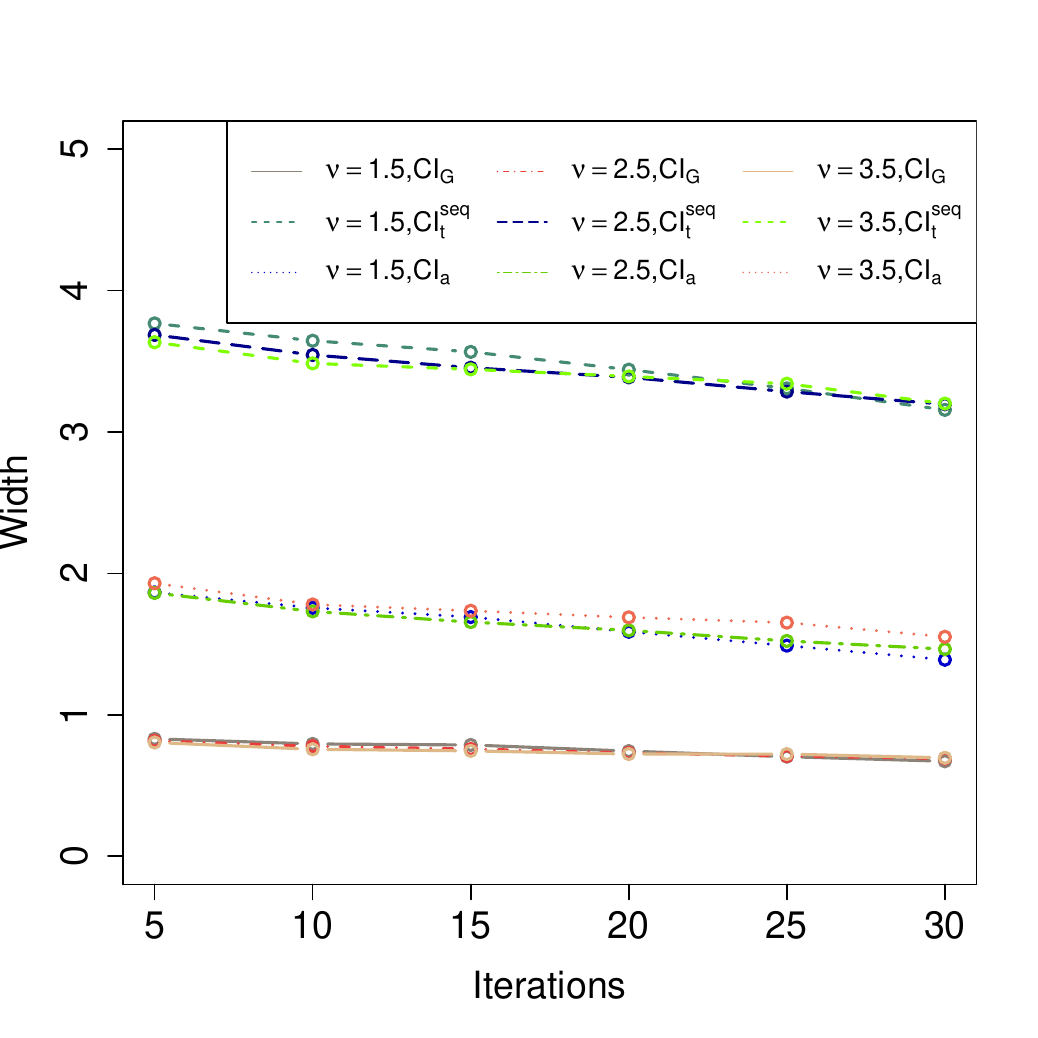}
  \caption{Coverage rates of $CI_t^{\mathbf{seq}}$ and $CI_G$ (Panel 1) and widths of $CI_t^{\mathbf{seq}}$, $CI_G$, and $CI_a$ (Panel 2) under different scenarios. The nominal confidence level is 95\%. The Gaussian process is well specified. More figures of misspecified cases are available in Appendix \ref{app:reGP}.}
  \label{figall}
\end{figure*}

\subsection{Misspecified Gaussian Process}\label{subsec:eg3}
Although our theory does not cover the case when the Gaussian process is misspecified, we conduct numerical studies to study the influence of model misspecification. The misspecification is in the sense that the correlation function is misspecified. Specifically, we consider the Mat\'ern correlation functions with $\nu_0=1.5,2.5,3.5$, and $A_0D_\Omega=25$. The rest of the settings are the same as those in Section \ref{subsec:eg1}. The only difference is that, we use Mat\'ern correlation functions with $\nu=1.5,2.5,3.5$ to construct predictors and confidence intervals for each $\nu_0$. The coverage rates of $CI_t^{\mathbf{seq}}$ and $CI_G$ and the width of $CI_t^{\mathbf{seq}}$, $CI_G$, and $CI_a$ are shown in Panels 1-6 in Figure \ref{figallmis} in Appendix \ref{app:reGP}. We find similar patterns as discussed in Section \ref{subsec:eg1}, namely, the proposed confidence interval is conservative while the naive one is permissive. We also find that for each $\nu_0$, the width of confidence intervals does not change a lot for different $\nu$. These findings indicate that our theory may work for the misspecified case, while the theoretical development needs further study.

We also consider another case of misspecification, where the correlation function used in the prediction is a rational quadratic kernel \citep{rasmussen2006gaussian} defined as
\begin{align*}
    \Psi_{QK}(x,x') = \left(1+\frac{\|x-x'\|^2}{2\alpha\phi}\right)^{-\alpha},
\end{align*}
where $\alpha,\phi>0$ are parameters. Here we consider $\alpha=1$ as in \cite{hennig2012entropy}. We choose $\phi=10$ and $20$. The results are in Figure \ref{figallmisQK} in Appendix \ref{app:reGP}. From Figure \ref{figallmisQK} it can be seen that both $CI_t^{\mathbf{seq}}$ and $CI_G$ do not achieve the nominal level, but $CI_t^{\mathbf{seq}}$ is much closer to the nominal level than the naive one. This indicates that if the misspecification is severe, then it has a strong influence of uncertainty quantification, while our method is more robust than the naive one. 

\subsection{Deterministic Functions}\label{subsec:eg2}
In this subsection, we consider three deterministic functions. Because a deterministic function is no longer random (thus not a Gaussian process), a model misspecification occurs. Furthermore, there is no definition of ``confidence interval'' for a deterministic function. Therefore, we evaluate the confidence intervals $CI_t^{\mathbf{seq}}$ and $CI_G$ by checking whether they cover the optimal point after certain number of iterations.

In both numerical examples in this subsection, we use $a_{\text{UCB}}$ (defined in Section \ref{Sec:existing}) as the acquisition function. The iteration numbers we consider are 5, 10, 15, 20, 25, 30. There are three parameters needed to be specified in the Gaussian process regression: the smoothness parameter $\nu$, the constant $A_0$ in Condition \ref{C1}, and the variance $\sigma^2$. Following the usual approaches in Gaussian process regression \citep{santner2013design}, we impose a specific $\nu$, and estimate $A_0$ and $\sigma^2$ via maximum likelihood estimation based on the initial evaluations of the function values on the initial points. These estimated parameters are used for constructing the prior distribution of underlying functions and evaluation of $a_{\text{UCB}}$ in Algorithm \ref{alg:bayesopt}, and constructing confidence intervals $CI_t^{\mathbf{seq}}$ (by Algorithm \ref{alg:upper}) and $CI_G$ (by \eqref{eq:naiveci}). For the conciseness, we put all the details and numerical results in this subsection to Appendix \ref{app:tablesDeter}, and only list the test functions we used in this section. Here we select three test functions from the Optimization category of Virtual Library of Simulation Experiemnts: Test function and Datasets (http://www.sfu.ca/~ssurjano/optimization.html).

$\bullet$ Modified test function in \cite{higdon2002space}: $$f_1(x)=1.5\sin(2\pi x/2)-0.2\sin(2\pi x/2.5)-(x-1)^2/120,$$ where $x\in [0,8]$. The modification is made because the original function is quite easy to be optimized. 


$\bullet$ The test function in  \cite{keane2008engineering}:
\begin{align*}
    f_3(x) = -(6x-2)^2\sin(12x-4), x\in [0,1].
\end{align*}

$\bullet$ The rescaled form of the Branin-Hoo function \citep{picheny2013benchmark}:
\begin{align*}
    f_4(x) = & -\frac{1}{51.95}\bigg(\left(\bar x_2-\frac{5.1\bar x_1^2}{4\pi^2}+\frac{5\bar x_1}{\pi}-6\right)^2\nonumber\\
    & +\left(10-\frac{10}{8\pi}\right)\cos(\bar x_1)-44.82\bigg),
\end{align*}
where $\bar x_1=15x_1-5, \bar x_2=15x_2$, and $x=(x_1,x_2)^T\in [0,1]^2$. This function has mean zero and variance one.

From Tables \ref{tab:eg2}-\ref{tab:eg5} in Appendix \ref{app:tablesDeter}, we can see that our proposed confidence interval is more robust than the naive confidence interval. Unlike misspecified Gaussian processes, the smoothness plays a more important role in the effectiveness of the confidence intervals. This suggests that when the underlying truth is a deterministic function, this kind of model misspecification has a strong impact on the quality of uncertainty quantification. Robust uncertainty quantification methodologies for deterministic functions will be pursued in the future.

\section{CONCLUSIONS AND DISCUSSION}\label{sec:dis}

In this work, we propose a novel methodology to construct confidence regions for the outputs given by any Bayesian optimization algorithm with theoretical guarantees. To the best of our knowledge, this is the \textit{first} result of this kind. As a cost of its high flexibility, the confidence regions may be somewhat conservative, because they are constructed based on generic probability inequalities that may not be tight enough. Nevertheless, given the fact that naive methods may be highly permissive, the proposed method can be useful when a conservative approach is preferred, such as in reliability assessments. To improve the power of the proposed method, one needs to seek for more accurate inequalities in a future work. One might also need to derive better error bounds tailored to specific acquisition functions and specify the constants in the upper bounds, and find robust confidence intervals, or other uncertainty quantification methods, which can mitigate the impact of model misspecification. Other possible future extensions include considering more surrogate models, such as Decision trees or Tree Parzen estimators.

\subsubsection*{Acknowledgements}
The authors are grateful to all reviewers for their helpful comments and suggestions. Tuo's research is supported by NSF DMS-1914636 and CCF-1934904.

\bibliography{ref}

\appendix

\numberwithin{equation}{section}
\numberwithin{table}{section}
\numberwithin{theorem}{section}
\numberwithin{lemma}{section}
\numberwithin{figure}{section}

\section{INEQUALITIES FOR GAUSSIAN PROCESSES}

In this section, we review some inequalities on the maximum of a Gaussian process. Let $G(x)$ be a separable zero-mean Gaussian process with $x\in \Gamma$. Define the metric on $\Gamma$ by 
\begin{align*}
    \mathfrak{d}_g(G(x_1), G(x_2))=\sqrt{\EE (G(x_1) - G(x_2))^2}.
\end{align*}
The $\epsilon$-covering number of the metric space $(\Gamma,\mathfrak{d}_g)$, denoted as $N(\epsilon,\Gamma,\mathfrak{d}_g)$, is the minimum integer $N$ so that there exist $N$ distinct balls in $(\Gamma,\mathfrak{d}_g)$ with radius $\epsilon$, and the union of these balls covers $\Gamma$. Let $D$ be the diameter of $\Gamma$ with respect to the metric $\mathfrak{d}_g$. The supremum of a Gaussian process is closely tied to a quantity called the \textit{entropy integral}, defined as
\begin{eqnarray}\label{entropy}
\int_0^{D/2} \sqrt{\log N(\epsilon,\Gamma,\mathfrak{d}_g)}d\epsilon.
\end{eqnarray}
For detailed discussion of entropy integral, we refer to \cite{adler2009random}. 

Lemma \ref{thm133} provides an upper bound on the expectation of the maximum value of a Gaussian process, which is Theorem 1.3.3 of \cite{adler2009random}. Note that the right-hand side of \eqref{thm133neq} is an upper bound of Talagrand's majorizing measure \citep{talagrand1996majorizing}. Talagrand's bound, albeit sharper, is not numerically tractable in the current context, because the covariance function of the GP regression posterior is highly nonstationary and complicated. Therefore, we apply Lemma \ref{thm133} in our proofs.
\begin{lemma}\label{thm133}
Let $G(x)$ be a separable zero-mean Gaussian process with $x$ lying in a $\mathfrak{d}_g$-compact set $\Gamma$, where $\mathfrak{d}_g$ is the metric. Let $N$ be the $\epsilon$-covering number. Then there exists a universal constant $\eta$ such that
\begin{align}\label{thm133neq}
    \mathbb{E}\left(\sup_{x\in \Gamma} G(x)\right) \leq \eta\int_0^{D/2} \sqrt{\log N(\epsilon,\Gamma,\mathfrak{d}_g)}d\epsilon.
\end{align}
\end{lemma}
Lemma \ref{thm211}, which is Theorem 2.1.1 of \cite{adler2009random}, presents a concentration inequality.
\begin{lemma}\label{thm211}
Let $G$ be a separable Gaussian process on a $\mathfrak{d}_g$-compact $\Gamma$ with mean zero, then for all $u>0$,
\begin{align}\label{thm211neq}
    \mathbb{P}\left(\sup_{x\in \Gamma} G(x) - \mathbb{E} (\sup_{x\in \Gamma} G(x)) > u\right) \leq e^{-u^2/2\sigma^2_\Gamma},
\end{align}
where $\sigma^2_\Gamma = \sup_{x\in \Gamma} \mathbb{E}G(x)^2$.
\end{lemma}

Theorem \ref{ourthmjasa} is a slightly strengthened version of Theorem 1 of \cite{wang2019prediction}. Its proof, in Section \ref{App:jasaprof}, is based on Lemmas \ref{thm133}-\ref{thm211} and some machinery from scattered data approximation \cite{wendland2004scattered}.

\begin{theorem}\label{ourthmjasa}
Suppose Condition \ref{C1} holds. Let $\mu(x)$ and $\sigma(x)$ be as in \eqref{mean} and \eqref{variance}, respectively, and $D_\Omega = \text{diam}(\Omega)$ be the Euclidean diameter of $\Omega$. Then for any $u>0$, and any closed deterministic subset $A\subset \Omega$, with probability at least $1-\exp\{-u^2/(2\sigma_A^2)\}$, the kriging prediction error has the upper bound
\begin{eqnarray}\label{eq:thmBound}
	\sup_{x\in A}Z(x)-\mu(x)\leq \eta_1 \sigma_A  \sqrt{p(1 \vee \log(A_0 D_\Omega))} \sqrt{\log(e\sigma/\sigma_A))} + u,
\end{eqnarray}
where $A_0$ is defined in Condition \ref{C1}, $\eta_1$ is a universal constant, and $\sigma_A = \sup_{x\in A}\sigma(x)$.
\end{theorem}

\section{PROOF OF THEOREM \ref{thm:undersmooth}}\label{sec:pfthm1}

We proof Theorem \ref{thm:undersmooth} by partitioning $\Omega$ into subregions, and applying Theorem \ref{ourthmjasa} on each of them.
Let $\Omega_i = \{x\in \Omega | \sigma e^{-i} \leqslant \sigma(x) \leqslant  \sigma e^{-i+1}\},$ for $i=1,\ldots$. Let $\sigma_i = \sup_{x\in \Omega_i} \sigma(x)$.

Take $\eta_2 = \eta_1 \sqrt{2}e$. By Theorem \ref{ourthmjasa}, we have
\begin{align*}
    & P\left(\sup_{x\in \Omega}\frac{Z(x) - \mu(x)}{\sigma(x)  \log^{1/2}(e\sigma/\sigma(x) )} >  \eta_2 \sqrt{p(1 \vee \log(A_0 D_\Omega))} + u \right) \nonumber\\
    \leqslant & \sum_{i=1}^\infty P\left(\sup_{x\in \Omega_i}\frac{Z(x) - \mu(x)}{\sigma(x)\log^{1/2}(e\sigma/\sigma(x) )}  >  \eta_2 \sqrt{p(1 \vee \log(A_0 D_\Omega))} + u \right)\nonumber\\
    \leqslant & \sum_{i=1}^\infty P\left(\sup_{x\in \Omega_i}Z(x) - \mu(x)  > ( \eta_2 \sqrt{p(1 \vee \log(A_0 D_\Omega))} + u)\sigma e^{-i}\sqrt{i} \right)\nonumber\\
    \leqslant & \sum_{i=1}^\infty P\left(\sup_{x\in \Omega_i}Z(x) - \mu(x)  > ( \eta_2 \sqrt{p(1 \vee \log(A_0 D_\Omega))} + u) \sigma_{i}\log^{1/2} (e\sigma/\sigma_{i})/(\sqrt{2}e) \right)\nonumber\\
    \leqslant & \sum_{i=1}^\infty \exp\left\{ - u^2\log (e\sigma/\sigma_{i})/(4 e^2)  \right\}\nonumber\\
    \leqslant & \sum_{i=1}^\infty \exp\left\{ - i u^2/(4 e^2)  \right\} = \frac{\exp\left\{ - u^2/(4 e^2)  \right\}}{1 - \exp\left\{ - u^2/(4 e^2)  \right\}},\nonumber
\end{align*}
which, together with the fact that $M\geq 0$, implies the following upper bound of $\mathbb{E}M$
\begin{align*}
    \mathbb{E} M & = \int_0^\infty \mathbb{P}(M > x) dx\nonumber\\
    & \leq \bigg(\int_0^{  \eta_2 \sqrt{p(1 \vee \log(A_0 D_\Omega))} + 1} + \int_{  \eta_2 \sqrt{p(1 \vee \log(A_0 D_\Omega))} + 1}^\infty\bigg)\mathbb{P}(M > x) dx\nonumber\\
    & \leq  \eta_2 \sqrt{p(1 \vee \log(A_0 D_\Omega))} + 1 + \int_1^\infty \frac{2\exp\left\{ - x^2/(4 e^2)  \right\}}{1 - \exp\left\{ - x^2/(4 e^2)  \right\}} dx\nonumber\\
    & \leq C_0 \sqrt{p(1 \vee \log(A_0 D_\Omega))}.
\end{align*}
To access the tail probability, we note that $M-\mathbb{E}M$ is also a Gaussian process with mean zero. Thus
by Lemma \ref{thm211}, we have
\begin{align*}
    \mathbb{P}(M-\mathbb{E}M>t)\leq  e^{-t^2/2\sigma^2_M},
\end{align*}
where $$\sigma^2_M = \sup_{x\in \Omega}\mathbb{E}\frac{(Z(x) - \mu(x))^2}{\sigma(x)^2\log (e\sigma/\sigma(x))} \leq 1.$$
Hence, we complete the proof.

\section{INDEPENDENCE IN SEQUENTIAL GAUSSIAN PROCESS MODELING}

The proof of Theorem \ref{thm:sequential} relies on certain independence properties of sequential Gaussian process modeling shown in Lemmas \ref{lemma:twosets}-\ref{thm:indenpendence}. First we introduce some notation.
For an arbitrary function $f$, and $X=(x_1,\ldots,x_n)$, define $f(X)=(f(x_1),\ldots,f(x_n))^T$, and
\begin{eqnarray}\label{interp}
\mathcal{I}_{\Psi,X}f( x)=r^T(x) K^{-1}  f(X),
\end{eqnarray}
where $r=(\Psi(x-x_1),\ldots,\Psi(x-x_n))^T, K=(\Psi(x_j-x_k))_{j k}$. For notational convenience, we define $\mathcal{I}_{\Psi,\varnothing}f=0$. 

\begin{lemma}\label{lemma:twosets}
Let $Z$ be a stationary Gaussian process with mean zero and correlation function $\Psi$. For two sets of scattered points $X'\subset X=(x_1,\ldots,x_n)$, we have
\begin{eqnarray}\label{decomposition}
Z-\mathcal{I}_{\Psi,X'}Z=(Z-\mathcal{I}_{\Psi,X}Z)+\mathcal{I}_{\Psi,X}(Z-\mathcal{I}_{\Psi,X'}Z).
\end{eqnarray}
In addition, if $X$ and $X'$ are deterministic sets, then the residual $Z-\mathcal{I}_{\Psi,X}Z$ and the vector of observed data $(Z(x_1),\ldots,Z(x_n))^T$ are mutually independent Gaussian process and vector, respectively.
\end{lemma}

\begin{proof}
It is easily seen that $\mathcal{I}_{\Psi,X}$ and $\mathcal{I}_{\Psi,X'}$ are linear operators and $\mathcal{I}_{\Psi,X'}\mathcal{I}_{\Psi,X}=\mathcal{I}_{\Psi,X}$, which implies \eqref{decomposition}.

The residual $Z-\mathcal{I}_{\Psi,X}Z$ is a Gaussian process because $\mathcal{I}_{\Psi,X}$ is linear. The independence between the Gaussian process and the vector can be proven by calculation the covariance
\begin{eqnarray*}
&&\text{Cov}(Z(x')-\mathcal{I}_{\Psi,X'}Z(x'),Z(X))\\
&=&\text{Cov}(Z(x')-r^T(x') K^{-1}  Z(X),Z(X))\\
&=& r(x')-r(x')=0,
\end{eqnarray*}
which completes the proof.
\end{proof}

\begin{lemma}\label{thm:indenpendence}
    For any instance algorithm of Bayesian optimization, the following statements are true.
    \begin{enumerate}
        \item Conditional on $\mathcal{F}_{n-1}$ and $X_{n}$, the residual process $Z(\cdot)-\mu_n(\cdot)$ is independent of $\mathcal{F}_n$.
        \item Conditional on $\mathcal{F}_{n}$, the residual process $Z(\cdot)-\mu_n(\cdot)$ is a Gaussian process with same law as $Z'(\cdot)-\mathcal{I}_{\Psi,X_{1:n}}Z'(\cdot)$, where $Z'$ is an independent copy of $Z$.
    \end{enumerate}
\end{lemma}

\begin{proof}
We use induction on $n$. For $n=1$, the desired results are direct consequences of Lemma \ref{lemma:twosets}, because the design set is suppressed conditional on $\mathcal{F}_0$.

Now suppose that we have proven already the assertion for $n$ and want to conclude it for $n+1$. First, we invoke the decomposition given by Lemma \ref{lemma:twosets} to have
\begin{eqnarray}\label{Zprime}
Z'-\mathcal{I}_{\Psi,X_{1:n}}Z'=(Z'-\mathcal{I}_{\Psi,X_{1:(n+1)}}Z')+\mathcal{I}_{\Psi,X_{1:(n+1)}}(Z'-\mathcal{I}_{\Psi,X_{1:n}}Z').
\end{eqnarray}
Because $\mu_n=\mathcal{I}_{\Psi,X_{1:n}}Z$, we also have
\begin{eqnarray}\label{Zmu}
Z-\mu_{n}=(Z-\mu_{n+1})+\mathcal{I}_{\Psi,X_{1:(n+1)}}(Z-\mu_n).
\end{eqnarray}
By the inductive hypothesis, $Z-\mu_n$ has the same law as $Z'-\mathcal{I}_{\Psi,X_{1:n}}Z'$ conditional on $\mathcal{F}_{n}$, denoted by $Z-\mu_n\eqd Z'-\mathcal{I}_{\Psi,X_{1:n}}Z'|\mathcal{F}_n$. 
Our assumption that $X_{n+1}$ is independent of $(Z,Z')$ conditional on $\mathcal{F}_n$ implies that $X_{n+1}$ is independent of $(Z-\mu_n,Z'-\mathcal{I}_{\Psi,X_{1:n}}Z')$ as well. Thus,
$$Z-\mu_n\eqd Z'-\mathcal{I}_{\Psi,X_{1:n}}Z'|\mathcal{F}_n,X_{n+1}. $$
Clearly, this equality in distribution is preserved by acting $\mathcal{I}_{\Psi,X_{1:(n+1)}}$ on both sides, which implies  $$\left(Z-\mu_n,\mathcal{I}_{\Psi,X_{1:(n+1)}}(Z-\mu_n)\right)\eqd \left(Z'-\mathcal{I}_{\Psi,X_{1:n}}Z',\mathcal{I}_{\Psi,X_{1:(n+1)}}(Z'-\mathcal{I}_{\Psi,X_{1:n}}Z')\right)|\mathcal{F}_n,X_{n+1}.$$ Incorporating the above equation with \eqref{Zprime} and \eqref{Zmu} yields 
\begin{eqnarray}\label{eqindist}
\left(Z-\mu_{n+1},Z-\mu_n)\right)\eqd \left(Z'-\mathcal{I}_{\Psi,X_{1:(n+1)}}Z',Z'-\mathcal{I}_{\Psi,X_{1:n}}Z'\right)|\mathcal{F}_n,X_{n+1}.
\end{eqnarray}
Now we consider the vectors $V:=Z(X_{n+1})-\mu_n(X_{n+1})$ and $V'=Z'(X_{n+1})-\mathcal{I}_{\Psi,X_{1:n}}Z'(X_{n+1})$. Then \eqref{eqindist} implies
\begin{eqnarray}\label{eqindist2}
\left(Z-\mu_{n+1},V)\right)\eqd \left(Z'-\mathcal{I}_{\Psi,X_{1:(n+1)}}Z',V'\right)|\mathcal{F}_n,X_{n+1}.
\end{eqnarray}
Because $V'$ consists of observed data, we can apply
Lemma \ref{lemma:twosets} to obtain that, conditional on $\mathcal{F}_n$ and $X_{n+1}$, $Z'-\mathcal{I}_{\Psi,X_{1:(n+1)}}Z'$ is independent of $V'$, which, together with \eqref{eqindist2}, implies that $Z-\mu_{n+1}$ and $V$ are independent conditional on $\mathcal{F}_n$ and $X_{n+1}$. Because $\mu_n(X_{n+1})$ is measurable with respect to the $\sigma$-algebra generated by $\mathcal{F}_n$ and $X_{n+1}$, we obtain that $Z-\mu_{n+1}$ is independent of $Z(X_{n+1})$ conditional on $\mathcal{F}_n$ and $X_{n+1}$, which proves Statement 1. Combining Statement 1 and \eqref{eqindist} yields Statement 2.
\end{proof}

\section{PROOF OF THEOREM \ref{thm:sequential}}\label{sec:pfthm2}

The law of total probability implies
\begin{eqnarray*}
&&\mathbb{P}(M_T-C \sqrt{p(1 \vee \log(A_0 D_\Omega))}>t)\\
&=&\sum_{i=n}^{\infty} \mathbb{P}(M_T-C \sqrt{p(1 \vee \log(A_0 D_\Omega))}>t|T=n)\mathbb{P}(T=n)\\&=&\sum_{n=1}^{\infty} \mathbb{P}(M_n-C \sqrt{p(1 \vee \log(A_0 D_\Omega))}>t|T=n)\mathbb{P}(T=n)\\
&=&\sum_{n=1}^{\infty} \mathbb{E}\left\{\mathbb{P}(M_n-C \sqrt{p(1 \vee \log(A_0 D_\Omega))}>t|\mathcal{F}_n)\Big|T=n\right\}\mathbb{P}(T=n),
\end{eqnarray*}
where the last equality follows from the fact that $\{T=n\}\in\mathcal{F}_n$, namely, $T$ is a stopping time. Clearly, the desired results are proven if we can show $\mathbb{P}(M_n-C \sqrt{p(1 \vee \log(A_0 D_\Omega))}>t|\mathcal{F}_n)<e^{-t^2/2}$. Now we resort to part 2 of Lemma \ref{thm:indenpendence}, which states that conditional on $\mathcal{F}_n$, $Z(\cdot)-\mu_n(\cdot)$ is identical in law to its independent copy $Z'(\cdot)-\mathcal{I}_{\Psi,X_{1:n}}Z'(\cdot)$. Although the event $\{M_n-C \sqrt{p(1 \vee \log(A_0 D_\Omega))}>t\}$ looks complicated, it is measurable with respect to $Z(\cdot)-\mu_n(\cdot)$. Thus, we arrive at
\begin{eqnarray}
&&\mathbb{P}(M_n-C \sqrt{p(1 \vee \log(A_0 D_\Omega))}>t|\mathcal{F}_n)\nonumber\\
&=&\mathbb{P}\left(\sup_{x\in \Omega}\frac{Z'(x) - \mathcal{I}_{\Phi,X_{1:n}}Z'(x)}{\sigma_n(x)\log^{1/2} (e \sigma/\sigma_n(x))}-C \sqrt{p(1 \vee \log(A_0 D_\Omega))}>t|\mathcal{F}_n\right).\label{condition}
\end{eqnarray}
Because $Z'$ is independent of $Z$, the part of conditioning with respect to $Z(X_{1:n})$ in \eqref{condition} has no effect on $Z'$. The only thing that matters is the effect of the conditioning on the design points $X_{1:n}$. Hence, \eqref{condition} is reduced to 
\begin{eqnarray}
\mathbb{P}\left(\sup_{x\in \Omega}\frac{Z'(x) - \mathcal{I}_{\Phi,X_{1:n}}Z'(x)}{\sigma_n(x)\log^{1/2} (e \sigma/\sigma_n(x))}-C \sqrt{p(1 \vee \log(A_0 D_\Omega))}>t|X_{1:n}\right).\label{probfixed}
\end{eqnarray}
Clearly, we can regard the points $X_{1:n}$ in the formula above as a fixed design. Then the probability \eqref{probfixed} is bounded above by $e^{-t^2/2}$ as asserted by Corollary \ref{coro:practice}.

\section{PROOF OF THEOREM \ref{ourthmjasa}}\label{App:jasaprof}

This proof is similar to Theorem 1 of \cite{wang2019prediction} but with a few technical improvements.

Because $\mu(x)$ is a linear combination of $Z(x_i)$'s, $\mu(x)$ is also a Gaussian process. The main idea of the proof is to invoke a maximum inequality for Gaussian processes, which states that the supremum of a Gaussian process is no more than a multiple of the integral of the covering number with respect to its natural distance $\mathfrak{d}$. See \cite{adler2009random,van1996weak} for related discussions.

Let $g(x) = Z(x)-\mu(x)$. For any $x,x'\in A$, because $A$ is deterministic, we have
\begin{align*}
\mathfrak{d}(x, x')^2 = &\mathbb{E}(g(x)-g(x'))^2\\
                      = & \mathbb{E}(Z(x)-\mu(x)-(Z(x')-\mu(x')))^2 \\
                      = &\sigma^2(\Psi(x-x) - r^T(x)K^{-1}r(x)  + \Psi(x'-x') - r^T(x')K^{-1}r(x')\\
          & -2[\Psi(x-x')- r^T(x')K^{-1}r(x) ]),
\end{align*}
where $r(\cdot) = (\Psi(\cdot-x_1),\ldots,\Psi(\cdot-x_n))^T$, $K = (\Psi(x_j-x_k))_{jk}$.

The rest of our proof consists of the following steps. In step 1, we bound the covering number $N(\epsilon,A,\mathfrak{d})$. Next we bound the diameter $D$. In step 3, we obtain a bound for the entropy integral. In the last step, we invoke Lemmas \ref{thm133} and \ref{thm211} to obtain the desired results.

\newpage

\noindent\textbf{Step 1: Bounding the covering number}

Let $h(\cdot) = \Psi(x-\cdot) - \Psi(x'-\cdot)$. It can verified that
\begin{align*}
\begin{split}
\mathfrak{d}(x, x')^2  = & - \sigma^2[h( x') - \mathcal{I}_{\Psi,X}h(x')]+\sigma^2[h(x) - \mathcal{I}_{\Psi,X}h( x)].
\end{split}
\end{align*}
By Theorem 11.4 of \cite{wendland2004scattered},
\begin{align}\label{eq:dist2}
\mathfrak{d}(x, x')^2 \leq 2 \sigma^2 (\sigma_A/\sigma\|h\|_{\mathcal{N}_\Psi(\mathbf{R}^d)}) = 2 \sigma \sigma_A\|h\|_{\mathcal{N}_\Psi(\mathbf{R}^d)},
\end{align}
where
\begin{align*}
    \sigma_A^2 = \sup_{x\in A}\sigma(x)^2 = \sigma^2\sup _{x\in A} (\Psi(x-x) - r^T(x)K^{-1}r(x)).
\end{align*}
Denote the Euclidean norm by $\|\cdot\|$.
Then, by the definition of the spectral density and the mean value theorem, we have
\begin{align}
\|h\|^2_{\mathcal{N}_\Psi(\mathbf{R}^d)} & = \Psi(x-x) - 2\Psi(x'-x) + \Psi(x'-x')\nonumber\\
& = 2\int_{\mathbf{R}^d}(1-\cos({(x-x')^T\omega}))\tilde \Psi(\bm\omega)d\bm\omega\nonumber\\
& \leq \bigg( 2\int_{\mathbf{R}^d}\|\omega\|\tilde \Psi(\bm\omega)d\bm\omega\bigg)\|x-x'\|\nonumber\\
&\leq 2A_0\|x-x'\|,\label{eq:entropyC1}
\end{align}
where the last inequality follows from the fact that $\|\omega\|\leq \|\omega\|_1$.
Combining \eqref{eq:dist2} and \eqref{eq:entropyC1} yields
\begin{align}\label{eq:dist}
\mathfrak{d}(x, x')^2 \leq 2A_0^{1/2} \sigma\sigma_A\|x-x'\|^{1/2}.
\end{align}
Therefore, the covering number is bounded above by
\begin{align}\label{e1}
\log N(\epsilon,A,\mathfrak{d}) \leq \log N\bigg(\frac{\epsilon^{4}}{4A_0 \sigma^2\sigma_A^2},A,\|\cdot\|\bigg).
\end{align}
The right side of \eqref{e1} involves the covering number of a Euclidean ball, which is well understood in the literature.
See Lemma 4.1 of \cite{pollard1990empirical}. This result leads to the bound
\begin{align}\label{eq:entropyBoundC1}
\log N(\epsilon,A,\mathfrak{d})\leq p\log\bigg(\frac{48 A_0 D_A \sigma^2\sigma_A^2}{\epsilon^4} + 1\bigg) \leq p\log\bigg(\frac{48A_0 D_\Omega \sigma^2\sigma_A^2}{\epsilon^4} + 1\bigg),
\end{align}
where $D_A = \text{diam}(A)$ and $D_\Omega = \text{diam}(\Omega)$ are the Euclidean diameter of $A$ and $\Omega$, respectively.

\noindent\textbf{Step 2: Bounding the diameter $D$}

Define the diameter under metric $\mathfrak{d}$ by $D = \sup_{x, x'\in A}\mathfrak{d}(x, x')$. For any $x,x'\in A$,
\begin{align}\label{eq:distD1}
\mathfrak{d}(x, x')^2 = &\mathbb{E}(g(x)-g(x'))^2 \leq  4\sup _{x\in A}\mathbb{E}(g(x))^2\nonumber\\
                      = & 4\sup _{x\in A}\mathbb{E}(Z(x)-\mathcal{I}_{\Psi,\mathbf{X}}Z(x))^2\nonumber\\
                      = & 4\sigma^2\sup _{x\in A} (\Psi(x-x) - r^T(x)K^{-1}r(x)) = 4 \sigma_A^2.
\end{align}
Thus we conclude that
\begin{align}\label{eq:diamC1}
D\leq 2 \sigma_A.
\end{align}

\newpage

\noindent\textbf{Step 3: Bounding the entropy integral}

By \eqref{eq:entropyBoundC1} and \eqref{eq:diamC1},
\begin{align}\label{eq:boundEC1}
\int_0^{D/2} \sqrt{\log N(\epsilon,A,\mathfrak{d})}d\epsilon & \leq \int_0^{\sigma_A} \sqrt{ p\log\bigg(\frac{48A_0 D_\Omega \sigma^2\sigma_A^{2}}{\epsilon^4} + 1\bigg)}d\epsilon\nonumber\\
& \leq \left(\int_0^{\sigma_A}d\epsilon \right)^{1/2}\left(\int_0^{\sigma_A} p\log\bigg(\frac{48A_0 D_\Omega \sigma^2\sigma_A^{2}}{\epsilon^4} + 1\bigg)d\epsilon\right)^{1/2}\nonumber\\
& = \left(\int_0^{\sigma_A}d\epsilon \right)^{1/2}\left(\sigma\int_0^{\sigma_A/\sigma} p\log\bigg(\frac{48A_0 D_\Omega \sigma_A^{2}}{u^4\sigma^2} + 1\bigg)du\right)^{1/2}\nonumber\\
& \leq \sigma_A^{1/2}\left(\sigma\int_0^{\sigma_A/\sigma} p\log\bigg(\frac{48A_0 D_\Omega \sigma_A^{2}}{u^4\sigma^2} + \frac{\sigma_A^{2}}{u^4\sigma^2}\bigg)du\right)^{1/2}\nonumber\\
& \leq \sqrt{2p} \sigma_A \sqrt{\log(e^2\sqrt{1 + 48A_0 D_\Omega}\sigma/\sigma_A))}\nonumber\\
& \leq \sqrt{4p} \sigma_A \sqrt{\log(e\sqrt{1 + 48A_0 D_\Omega})} \sqrt{\log(e\sigma/\sigma_A))}\nonumber\\
& \leq c\sqrt{p(1 \vee \log(A_0 D_\Omega))} \sigma_A \sqrt{\log(e\sigma/\sigma_A))},
\end{align}
where $c = \sqrt{6\log(7e)}$.

\noindent\textbf{Step 4: Bounding $\mathbb{P}(\sup_{x\in A} Z(x)-\mu(x) > \eta\int_0^{D/2} \sqrt{\log N(\epsilon,A,\mathfrak{d})}d\epsilon + u)$}

By Lemmas \ref{thm133} and \ref{thm211}, we have 
\begin{align}\label{BorellTIS}
P\bigg(\sup_{x\in A} Z(x)-\mu(x) > \eta\int_0^{D/2} \sqrt{\log N(\epsilon,A,\mathfrak{d})}d\epsilon + t\bigg)\leq e^{-t^2/(2\sigma^2_A)}.
\end{align}
By plugging \eqref{eq:boundEC1} into \eqref{BorellTIS}, we obtain the desired inequality with $\eta_1 = c\eta$, which completes the proof.

\section{DETAILS OF CALIBRATING $C$ VIA SIMULATION}\label{App:simu}

An upper bound of the constant $C$ in Theorem \ref{thm:undersmooth} can be obtained by examine the proof of Lemma \ref{thm133} and Theorem \ref{ourthmjasa}. However, this theoretical upper bound can be too large for practical use. In this section, we consider estimating $C$ via numerical simulation. 

According to Part 1 of Theorem \ref{thm:undersmooth}, $$C_0 \geq \mathbb{E}M/\sqrt{p(1 \vee \log(A_0 D_\Omega))},$$
where $M=\sup_{x\in \Omega}\frac{Z(x) - \mu(x)}{\sigma(x)\log^{1/2} (e \sigma/\sigma(x))}$, $A_0$ is as in \eqref{A0}, and $D_\Omega$ is the Euclidean diameter of $\Omega$. 
For a specific Gaussian process, $\mathbb{E}M/\sqrt{p(1 \vee \log(A_0 D_\Omega))}$ is a constant and can be obtained by Monte Carlo. Let $\mathcal{M}$ be the collection of Gaussian processes satisfying the conditions of Theorem \ref{thm:undersmooth}. Then
$$
C_0 = \sup_{M \in\mathcal{M}}\mathbb{E}M/\sqrt{p(1 \vee \log(A_0 D_\Omega))} =: \sup_{M \in\mathcal{M}} H(M) .
$$
 The idea is to consider various Gaussian processes and find the maximum value of $\mathbb{E}M/\sqrt{p(1 \vee \log(A_0 D_\Omega))}$. This value can be close to $C$ when we cover a broad range of Gaussian processes.

In the numerical studies, we consider $\Omega=[0,1]^p$ for $p=1,2,3$. We consider different $A_0$ values to get different $A_0D_\Omega$'s. In each Monte Carlo sampling, we approximate $M$ using $$M_1 = \sup_{x\in \Omega_1}\frac{Z(x) - \mu(x)}{\sigma(x)\log^{1/2} (e \sigma/\sigma(x))},$$ where $\Omega_1$ is the first $100, 1000, 2000$ points of the Halton sequence \citep{niederreiter1992random} for $p=1,2,3$, respectively. We calculate the average of $M_1/\sqrt{p(1 \vee \log(A_0 D_\Omega))}$ over all the simulated realizations of each Gaussian process. 

Specifically, We simulate $1000$ realizations of the Gaussian processes for $p=1$, $100$ realizations for $p=2,3$ and consider the following four cases.  In Cases 1-3, we use maximin Latin hypercube designs \citep{santner2013design}, and use independent samples from the uniform distribution in Case 4.

\textbf{Case 1:} We consider $p=1$ with $20$ and $50$ design points. We consider the Gaussian correlation functions and Mat\'ern correlation functions with $\nu = 1.5,2.5,3.5$. The results are presented in Table \ref{p1tab}.

\textbf{Case 2:} We consider $p=2$ with $20$, $50$, and $100$ design points. We consider the Gaussian correlation functions and product Mat\'ern correlation functions with $\nu = 1.5,2.5,3.5$. The results are presented in Table \ref{p2tab}.

\textbf{Case 3:} 
We consider $p=3$ with $20$, $50$, $100$ and $500$ design points. We consider the product Mat\'ern correlation functions with $\nu = 1.5,2.5,3.5$. The results are shown in Table \ref{p3table}.

\textbf{Case 4:} We consider $p=2$ with $20$, $50$, and $100$ design points. We consider the product Mat\'ern correlation functions with $\nu = 1.5,2.5,3.5$.
The results are shown in Table \ref{uniftab}.

\begin{table}[h!]
  \caption{Simulation Results of Case 1}
  \label{p1tab}
  \centering
  \begin{tabular}{l|llllllll}
    \toprule
   & design points      &   $A_0 D_\Omega=1$   & $A_0 D_\Omega=3$  & $A_0 D_\Omega=5$  &   $A_0 D_\Omega=10$  &   $A_0 D_\Omega=25$\\
    \midrule
   Gaussian & 20 & 0.11640290 &  0.1978563 & 0.2450737 & 0.4542654 & 0.859318 \\
    & 50 & 0.08102775    & 0.0916648 & 0.1206034   & 0.1683377 & 0.422786\\
    \midrule
    $\nu = 1.5$ & 20 & 0.9640650 &  1.065597 & 0.9537634 & 0.9429957 & 1.0197966\\
    & 50 & 0.9442937    & 1.009187 & 0.8981430   & 0.8331926 & 0.8372607\\
    \midrule
    $\nu = 2.5$ & 20 & 0.7432965 &  0.8554707 & 0.7804686 & 0.8371662 & 1.0074204 \\
    & 50 & 0.7304104    & 0.8218710  & 0.7346077   & 0.6987832  & 0.7563067\\
    \midrule
    $\nu = 3.5$ & 20 & 0.6054239 &  0.7248086 & 0.6833789 & 0.7711124 & 0.9608837\\
    & 50 & 0.3367513    & 0.6941391 & 0.6244660   & 0.6278185 & 0.6928741\\
    \bottomrule
  \end{tabular}
\end{table}

\begin{table}[h!]
  \caption{Simulation Results of Case 2}
  \label{p2tab}
  \centering
  \begin{tabular}{l|llllllll}
    \toprule
   & design points      &   $A_0 D_\Omega=1$   & $A_0 D_\Omega=3$  & $A_0 D_\Omega=5$  &   $A_0 D_\Omega=10$  &   $A_0 D_\Omega=25$\\
    \midrule
   Gaussian & 20 & 0.2801128 &  0.4767259 & 0.5644628  &   0.7408401 & 1.0554507 \\
    & 50 & 0.1465512    & 0.2927036 &  0.3789438  &   0.5683807 & 0.9309326\\
    & 100 & 0.1156139    & 0.1961319 & 0.2436626   &   0.4189444  & 0.7641615\\
    \midrule
    $\nu = 1.5$ & 20 & 0.8106718 &  0.9528429 & 0.8748865 & 0.9365989 & 1.0894451\\
    & 50 & 0.8114071    & 0.9299506 & 0.8568070   & 0.8576984 & 0.9964256\\
    & 100 & 0.8137517    & 0.9108342 & 0.8224467   & 0.7951887 & 0.9168643\\
    \midrule
    $\nu = 2.5$ & 20 & 0.6072854 &  0.7709362 & 0.7411921 & 0.8540687 & 1.0933120\\
    & 50 & 0.6316136    & 0.7218077  & 0.7218077 &  0.7690956 & 0.9703693\\
    & 100 &  0.5651732   & 0.6677120 & 0.6677120  & 0.7090934 & 0.8791792\\
    \midrule
    $\nu = 3.5$ & 20 & 0.5243251 &  0.6881401 & 0.6915576 & 0.8290974 & 1.0876019\\
    & 50 & 0.3947094    & 0.6420423 & 0.6434791 & 0.7030224 & 0.9494486\\
    & 100 & 0.2898865    & 0.6279639 & 0.6036111 & 0.6420049 & 0.8373886\\
    \bottomrule
  \end{tabular}
\end{table}

\begin{table}[h!]
  \caption{Simulation Results of Case 3}
  \label{p3table}
  \centering
  \begin{tabular}{lll}
    \toprule
   Cases     & &   $H(M)$    \\
    \midrule
    20 design points, $\nu=1.5$, $A_0 D_\Omega = 1$ & & 0.6977030 \\
    500 design points, $\nu=3.5$, $A_0 D_\Omega = 5$ & & 0.4961581 \\
    100 design points, $\nu=2.5$, $A_0 D_\Omega = 3$ & & 0.6628567 \\
    50 design points, $\nu=1.5$, $A_0 D_\Omega = 10$ & & 0.7632713 \\
    \bottomrule
  \end{tabular}
\end{table}

\begin{table}[h!]
\caption{Simulation Results of Case 4}
\label{uniftab}
\centering
  \begin{tabular}{lll}
    \toprule
   Cases      & &  $H(M)$      \\
    \midrule
    100 design points, $\nu=3$, $A_0 D_\Omega = 3$ & & 0.6778535 \\
    50 design points, $\nu=1.5$, $A_0 D_\Omega = 1$ & & 0.8144700 \\
    20 design points, $\nu=2.5$, $A_0 D_\Omega = 5$ & & 0.7735112 \\
    100 design points, $\nu=1.5$, $A_0 D_\Omega = 10$ & & 0.8164859 \\
    \bottomrule
  \end{tabular}
\end{table}
From Tables \ref{p1tab}-\ref{uniftab}, we find the following patterns:
\begin{itemize}
    \item All numerical values ($H(M)$) in Tables \ref{p1tab}-\ref{uniftab} are less than 1.10. Only eight entries are greater than one.
    \item In most scenarios, the obtained values are decreasing in $\nu$. This implies that $H(M)$ is smaller when $M$ is smoother.
    \item $H(M)$ is not monotonic in $A_0D_\Omega$, which implies a more complicated function relationship between $H(M)$ and $A_0D_\Omega$.
    \item In most scenarios, $H(M)$ decreases as the dimension $p$ increases.
    \item The obtained values are decreasing in the number of design points.
\end{itemize}
In summary, the largest $H(M)$ values are observed when the sample size is small, the smoothness is low and the dimension is low. Therefore, we believe that our simulation study covers the largest possible $H(M)$ values and our suggestion of choosing $C_0=1$ can be used in most practical situations.

\section{MORE FIGURES OF NUMERICAL EXPERIMENTS FOR MISSPECIFIED GAUSSIAN PROCESSES}\label{app:reGP}
Here we present more figures of numerical experiments when Gaussian process is misspecified, as shown in Figures \ref{figallmis} and \ref{figallmisQK}.

\begin{figure}[h!]
  \centering
  \includegraphics[width=0.4\textwidth]{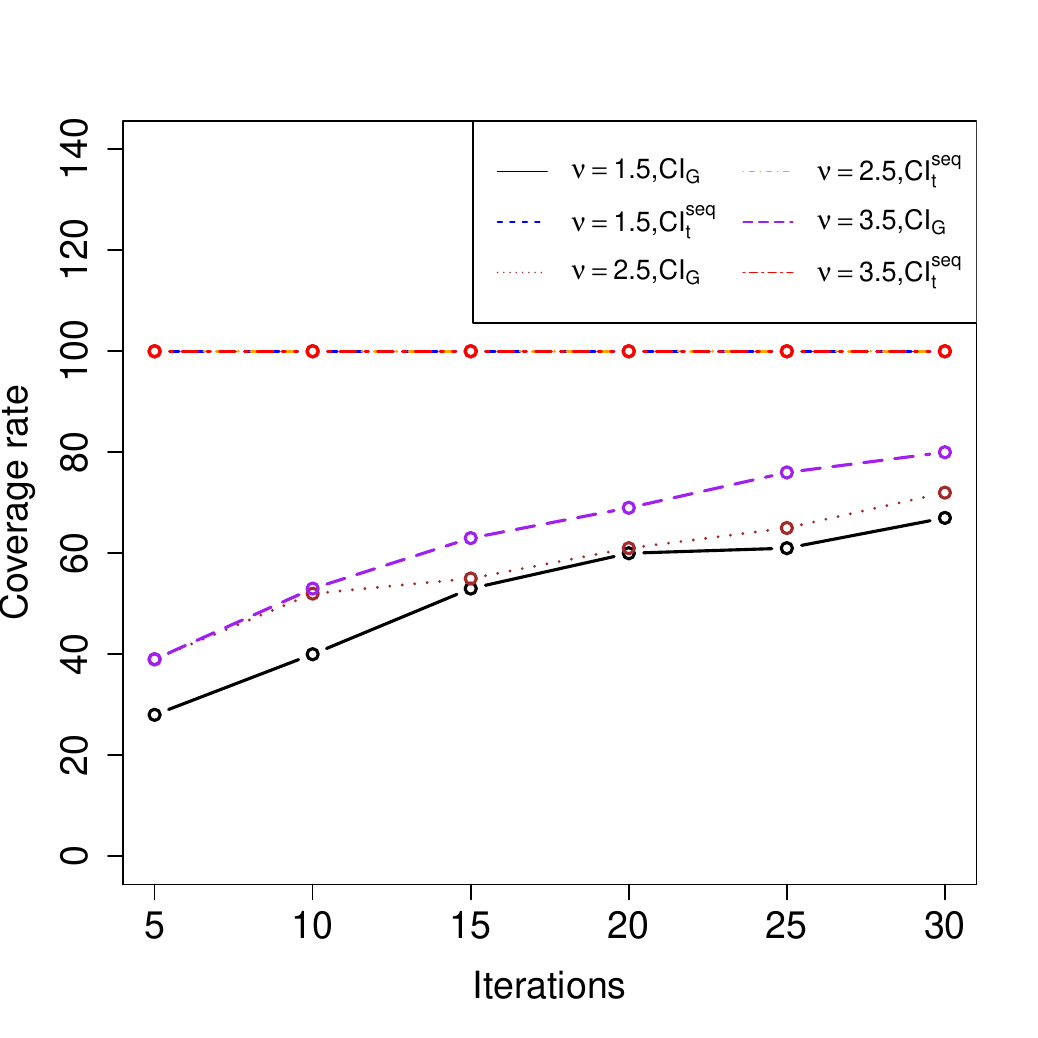}
        \includegraphics[width=0.4\textwidth]{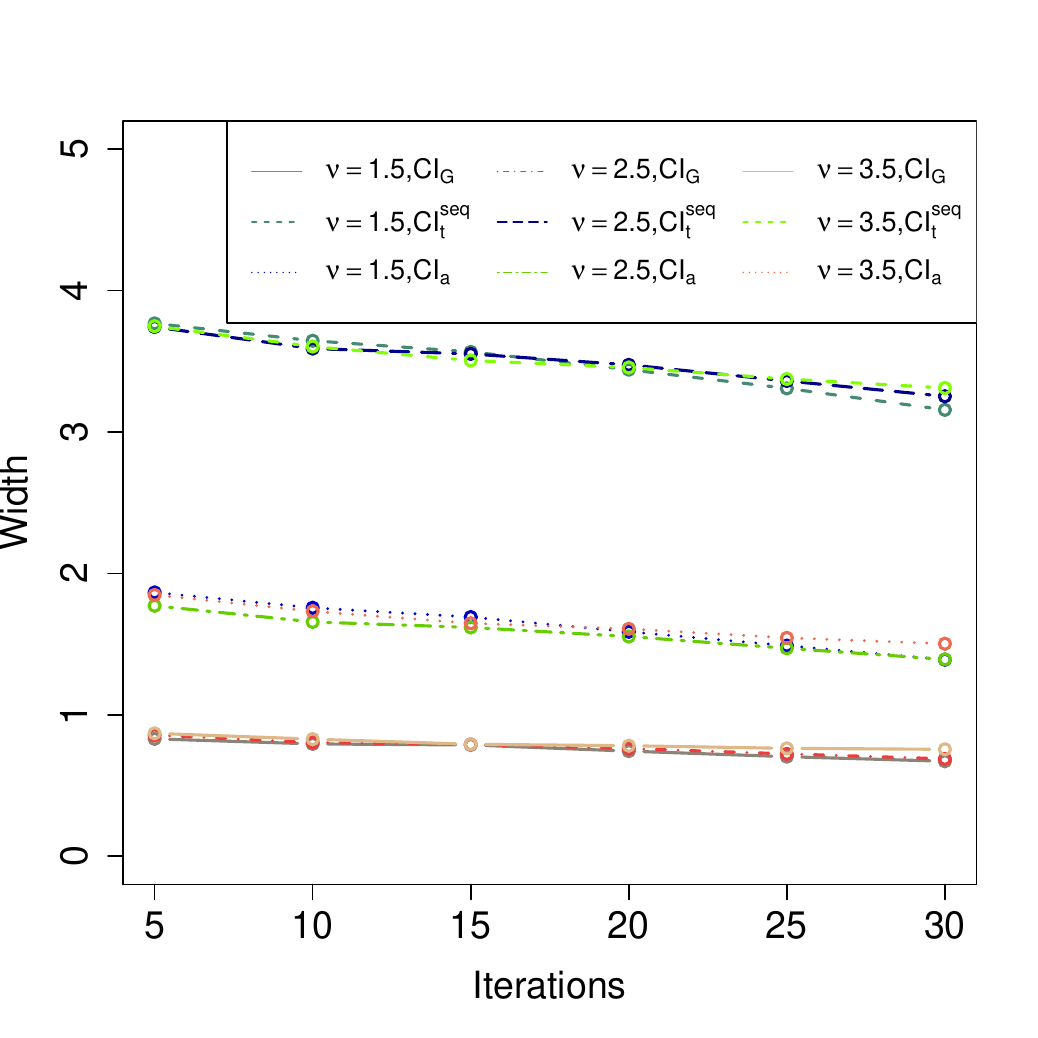}
        \includegraphics[width=0.4\textwidth]{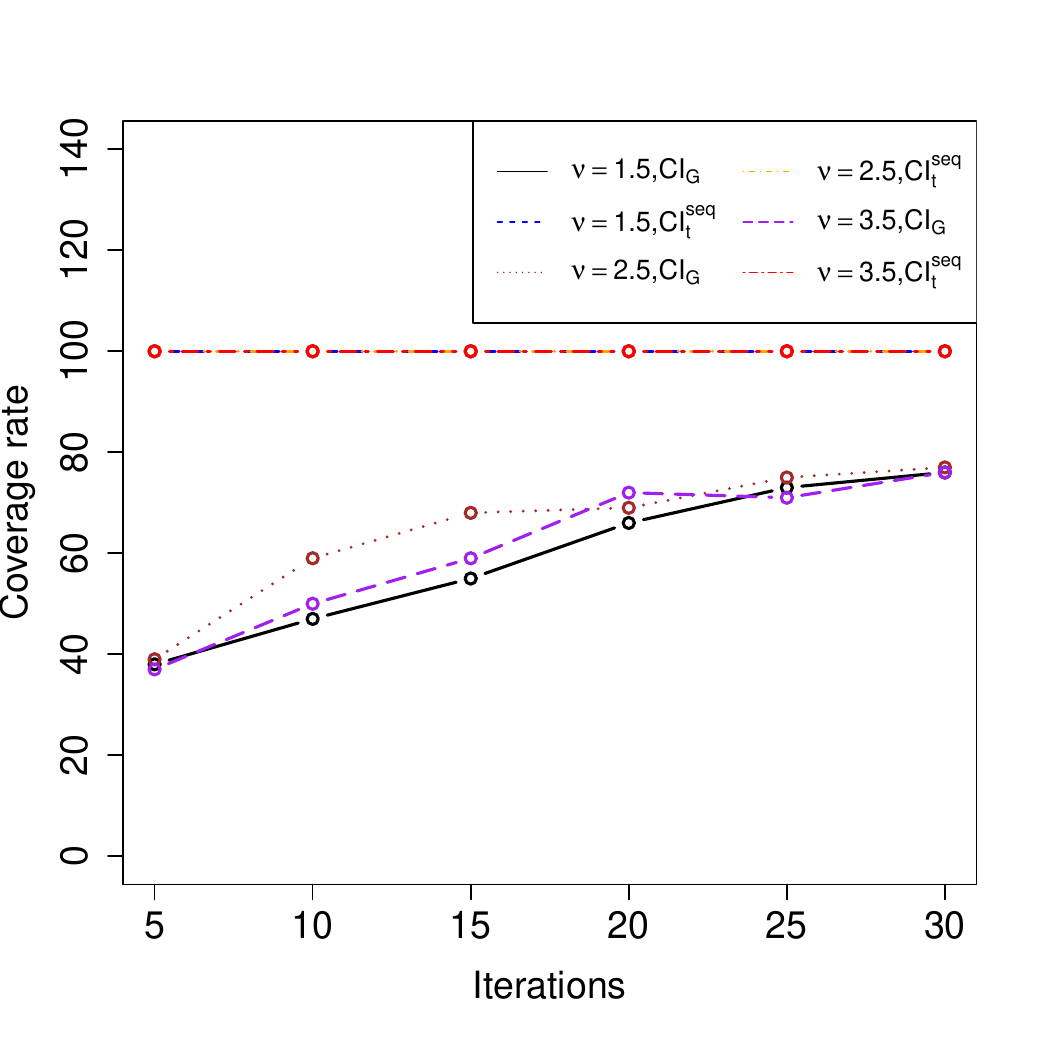}
        \includegraphics[width=0.4\textwidth]{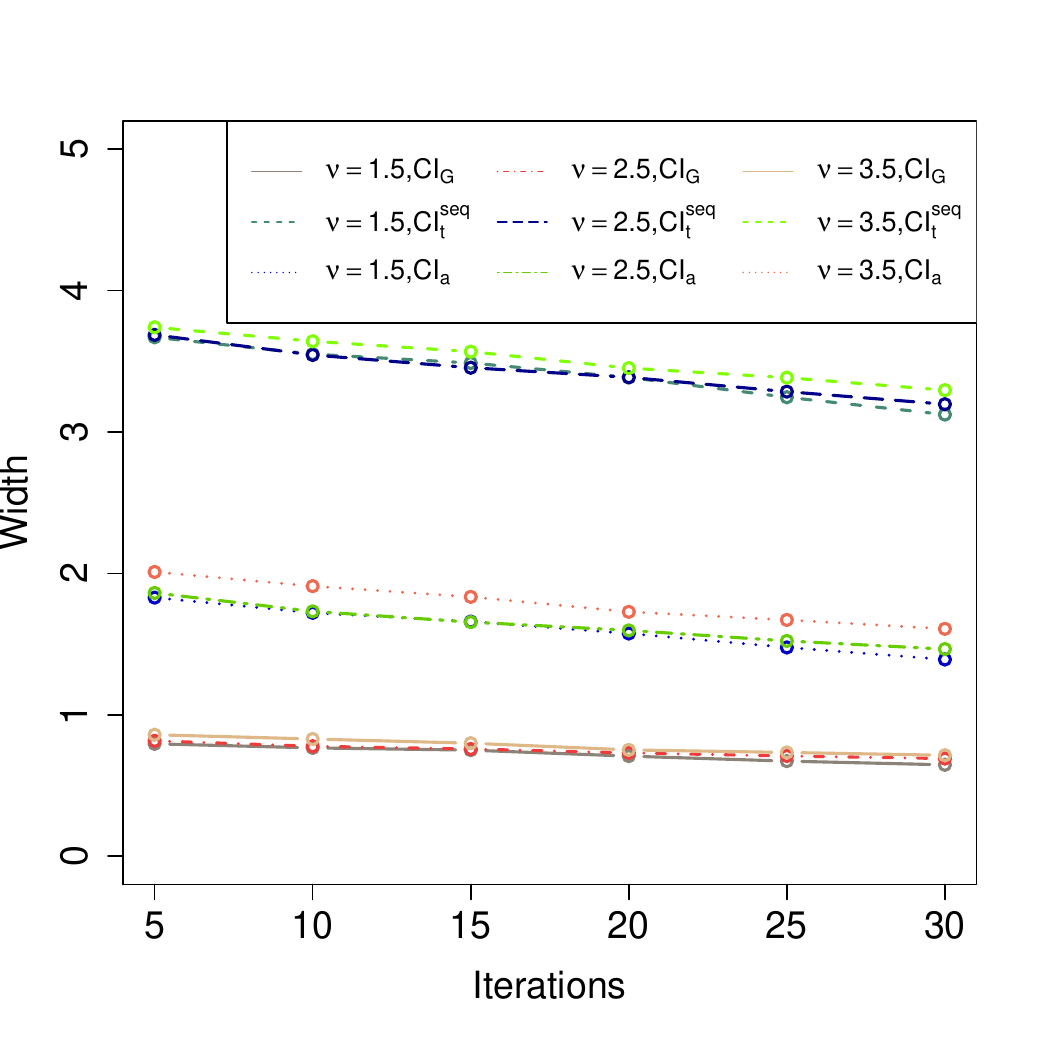}
        \includegraphics[width=0.4\textwidth]{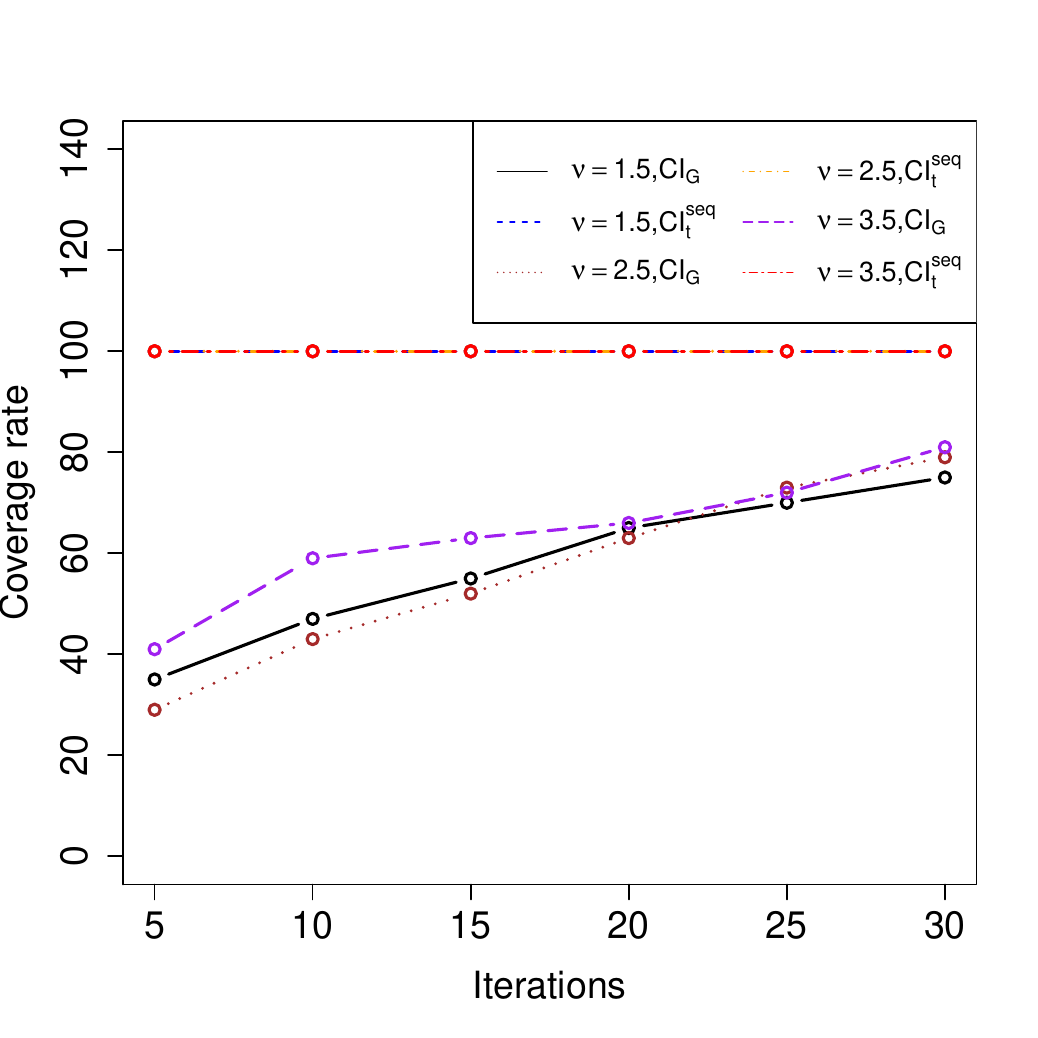}
        \includegraphics[width=0.4\textwidth]{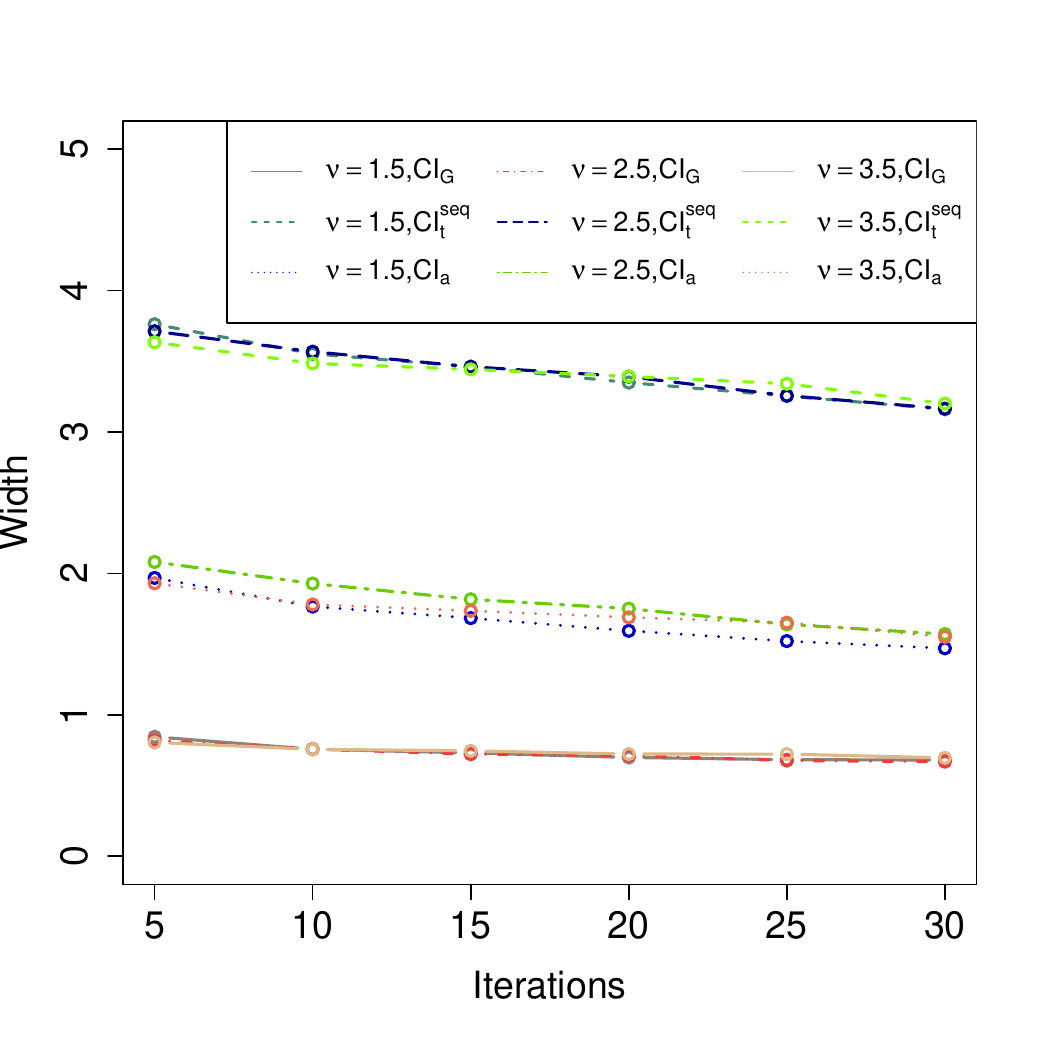}
  \caption{Coverage rates of $CI_t^{\mathbf{seq}}$ and $CI_G$ (Panels 1, 3, 5, 7) and widths of $CI_t^{\mathbf{seq}}$, $CI_G$, and $CI_a$ (Panels 2, 4, 6, 8) under different scenarios. The nominal confidence level is 95\%. \textbf{Panels 1 and 2:} The Gaussian process is well specified. 
  \textbf{Panels 1 and 2:} The Gaussian process is misspecified with $\nu_0=1.5$.
\textbf{Panels 3 and 4:} The Gaussian process is misspecified with $\nu_0=2.5$.
\textbf{Panels 5 and 6:} The Gaussian process is misspecified with $\nu_0=3.5$.}
  \label{figallmis}
\end{figure}

\begin{figure}[h!]
  \centering
  \includegraphics[width=0.45\textwidth]{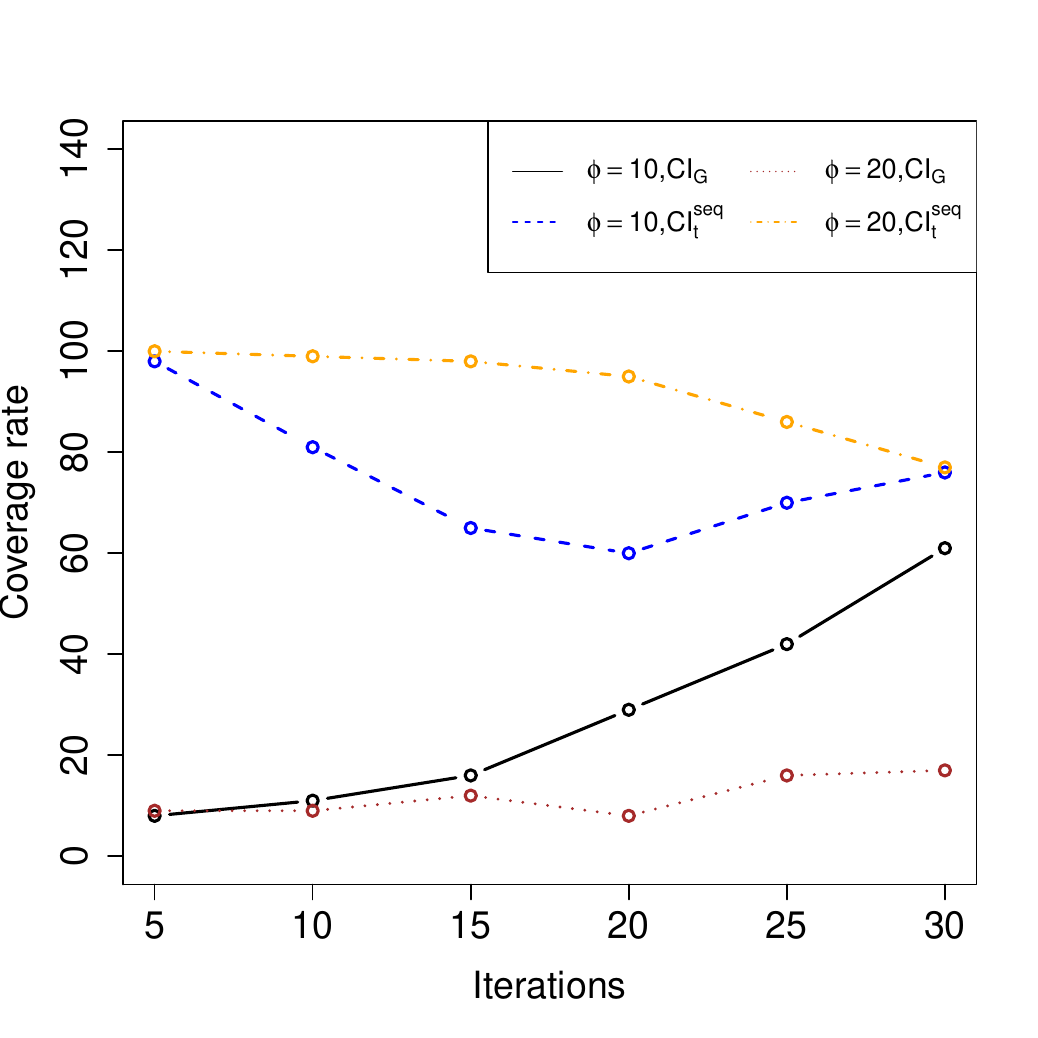}
        \includegraphics[width=0.45\textwidth]{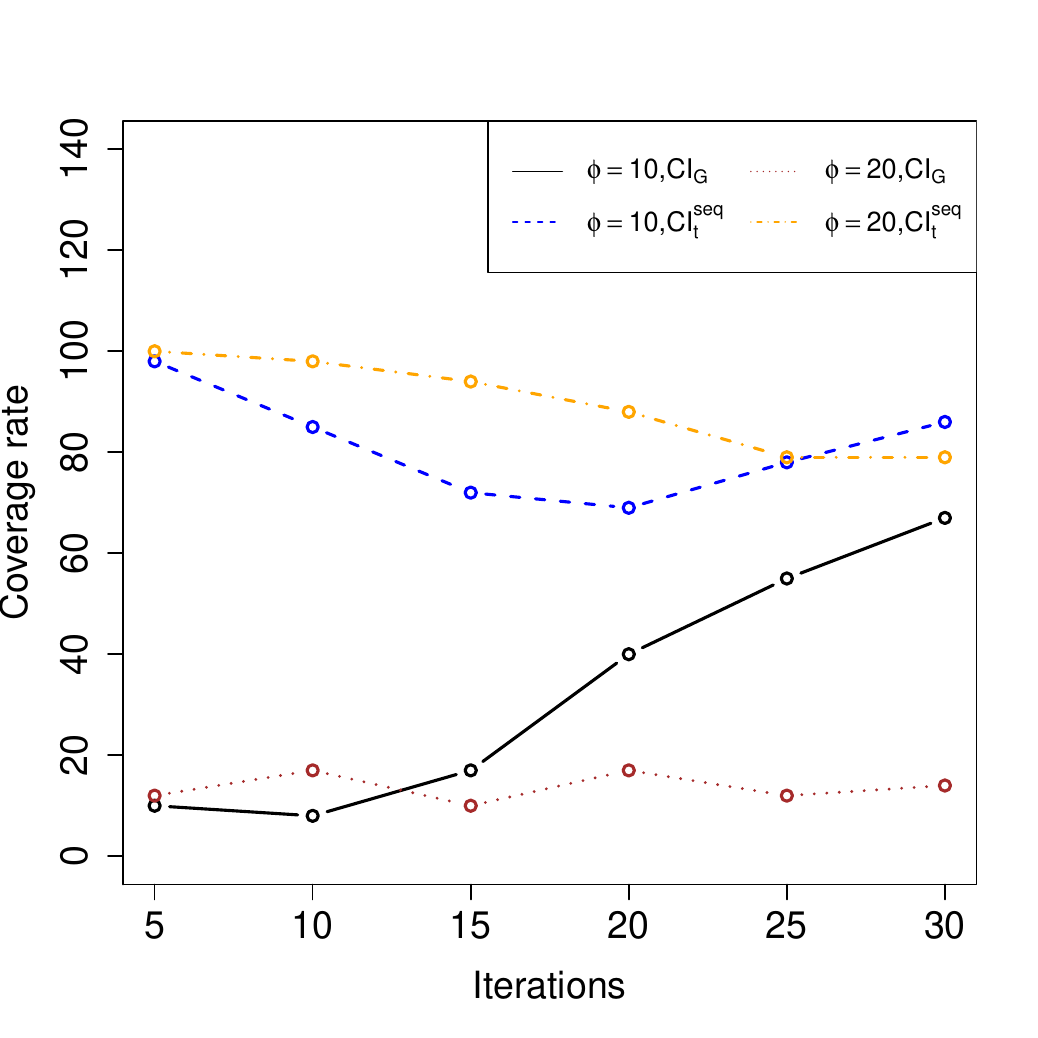}
        \includegraphics[width=0.45\textwidth]{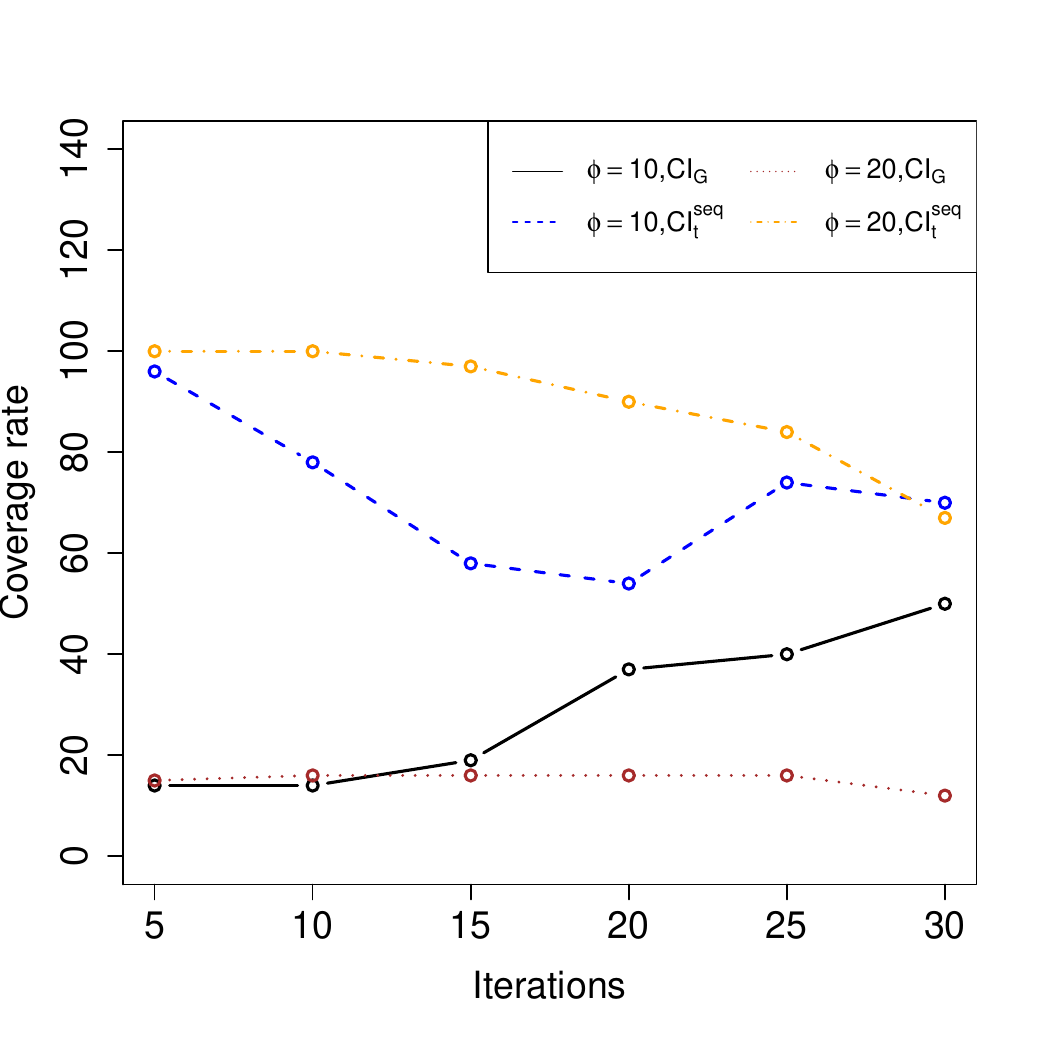}
  \caption{Coverage rates of $CI_t^{\mathbf{seq}}$ and $CI_G$ under different scenarios, where a rational quadratic correlation function is used for prediction. The nominal confidence level is 95\%. The underlying true correlation function is Mat\'ern with smoothness parameter $\nu_0$. \textbf{Panel 1:} $\nu_0=1.5$.
\textbf{Panel 2:} $\nu_0=2.5$.
\textbf{Panel 2:} $\nu_0=3.5$.}
  \label{figallmisQK}
\end{figure}

\section{DETAILS OF NUMERICAL EXPERIMENTS FOR DETERMINISTIC FUNCTIONS}\label{app:tablesDeter}

\paragraph{Deterministic function 1} The first deterministic function we consider is $$f_1(x)=1.5\sin(2\pi x/2)-0.2\sin(2\pi x/2.5)-(x-1)^2/120,$$ where $x\in [0,8]$, which is a modification of the function used in \cite{higdon2002space}. The modification is made because the original function is quite easy to be optimized. The maximum of $f_1$ is taken on the point $x^*=0.520$, and the maximum is $f_1(x^*)=1.5953$. The initial points are selected as 30 equally spaced points on the interval $[0,8]$. The results are collected in Table \ref{tab:eg2} in Appendix \ref{app:tablesDeter}.


\paragraph{Deterministic function 2} The third deterministic function we consider is the test function in  \cite{keane2008engineering}:
\begin{align*}
    f_3(x) = -(6x-2)^2\sin(12x-4), x\in [0,1].
\end{align*}
The maximum of $f_3$ is taken on the point $x^* = 0.7575$, with $f_3(x^*)=6.0207$. The initial points are selected as 30 equally spaced points on the interval $[0,8]$. We use $\Omega_1$ to approximate $\Omega$, where $\Omega_1$ is a set of grid points with grid length $1/2499$ (thus, there are 2500 test points in total). 

\paragraph{Deterministic function 3} The fourth deterministic function we consider is the rescaled form of the Branin-Hoo function \citep{picheny2013benchmark}:
\begin{align*}
    f_4(x) = & -\frac{1}{51.95}\bigg(\left(\bar x_2-\frac{5.1\bar x_1^2}{4\pi^2}+\frac{5\bar x_1}{\pi}-6\right)^2 +\left(10-\frac{10}{8\pi}\right)\cos(\bar x_1)-44.82\bigg),
\end{align*}
where $\bar x_1=15x_1-5, \bar x_2=15x_2$, and $x=(x_1,x_2)^T\in [0,1]^2$. The settings are the same of that in Deterministic function 2.

\begin{table}[h!]
\caption{Simulation Results of Deterministic Function 1. The following abbreviations are used: IN = iteration number, CI = confidence interval. The notation \Checkmark stands for ``cover'' and \XSolidBrush stands for ``not cover''.}\label{tab:eg2}
  \centering
  \begin{tabular}{l|lllllllll}
    \toprule
   & CI    &  IN $= 5$   & IN $= 10$  & IN $= 15$  &   IN $= 20$  &   IN $= 25$ & IN $= 30$\\
    \midrule
    $\nu = 1.5$ & $CI_t^{\mathbf{seq}}$ & \Checkmark &  \Checkmark & \Checkmark & \Checkmark & \Checkmark & \Checkmark \\
    & $CI_G$ & \Checkmark    & \Checkmark & \Checkmark   & \Checkmark & \Checkmark & \Checkmark\\
    \midrule
    $\nu = 2.5$ & $CI_t^{\mathbf{seq}}$ & \Checkmark &  \Checkmark & \Checkmark & \Checkmark & \Checkmark & \Checkmark \\
    & $CI_G$ & \Checkmark    & \XSolidBrush & \XSolidBrush   & \XSolidBrush & \XSolidBrush & \XSolidBrush\\
    \midrule
    $\nu = 3.5$ & $CI_t^{\mathbf{seq}}$ & \Checkmark &  \Checkmark & \Checkmark & \Checkmark & \Checkmark & \Checkmark \\
    & $CI_G$ & \Checkmark    & \XSolidBrush & \XSolidBrush   & \XSolidBrush & \XSolidBrush & \XSolidBrush\\
    \midrule
    $\nu = 5.5$ & $CI_t^{\mathbf{seq}}$ & \XSolidBrush & \XSolidBrush & \XSolidBrush   & \XSolidBrush & \XSolidBrush & \XSolidBrush\\ 
    & $CI_G$ & \XSolidBrush & \XSolidBrush & \XSolidBrush & \XSolidBrush   & \XSolidBrush & \XSolidBrush\\ 
    \bottomrule
  \end{tabular}
\end{table}


\begin{table}[h!]
  \caption{Simulation Results of Deterministic Function 2. The following abbreviations are used: IN = iteration number, CI = confidence interval. The notation \Checkmark stands for ``cover'' and \XSolidBrush stands for ``not cover''.}  \label{tab:eg4}
  \centering
  \begin{tabular}{l|llllllll}
    \toprule
   & CI    &  IN $= 5$   & IN $= 10$  & IN $= 15$  &   IN $= 20$  &   IN $= 25$ &   IN $= 30$\\
    \midrule
    $\nu = 1.5$ & $CI_t^{\mathbf{seq}}$ & \Checkmark &  \Checkmark & \Checkmark & \Checkmark  & \Checkmark & \Checkmark\\
    & $CI_G$  & \Checkmark &  \Checkmark & \Checkmark & \Checkmark  & \Checkmark & \Checkmark \\
    \midrule
    $\nu = 2.5$ & $CI_t^{\mathbf{seq}}$ & \Checkmark &  \Checkmark & \Checkmark & \Checkmark & \Checkmark & \Checkmark \\
    & $CI_G$ & \Checkmark &  \Checkmark & \Checkmark & \Checkmark  & \Checkmark & \Checkmark\\
    \midrule
    $\nu = 4$ & $CI_t^{\mathbf{seq}}$ & \Checkmark &  \Checkmark & \Checkmark & \Checkmark  & \Checkmark & \Checkmark\\
    & $CI_G$ & \Checkmark &  \Checkmark & \XSolidBrush & \XSolidBrush  & \XSolidBrush & \XSolidBrush\\
    \midrule
    $\nu = 5.5$  & $CI_t^{\mathbf{seq}}$ & \Checkmark &  \Checkmark & \Checkmark & \Checkmark  & \Checkmark & \Checkmark\\
    & $CI_G$ & \Checkmark &  \Checkmark & \XSolidBrush & \XSolidBrush  & \XSolidBrush & \XSolidBrush\\
    \bottomrule
  \end{tabular}
\end{table}

\begin{table}[h!]
  \caption{Simulation Results of Deterministic Function 3. The following abbreviations are used: IN = iteration number, CI = confidence interval. The notation \Checkmark stands for ``cover'' and \XSolidBrush stands for ``not cover''.}  \label{tab:eg5}
  \centering
  \begin{tabular}{l|llllllll}
    \toprule
   & CI    &  IN $= 5$   & IN $= 10$  & IN $= 15$  &   IN $= 20$  &   IN $= 25$ &   IN $= 30$\\
    \midrule
    $\nu = 1.1$ & $CI_t^{\mathbf{seq}}$ & \Checkmark &  \Checkmark & \Checkmark & \Checkmark  & \Checkmark & \Checkmark\\
    & $CI_G$  & \Checkmark & \Checkmark    & \Checkmark   & \Checkmark & \Checkmark & \Checkmark\\
    \midrule
    $\nu = 2.3$ & $CI_t^{\mathbf{seq}}$ & \Checkmark &  \Checkmark & \Checkmark & \Checkmark & \Checkmark & \Checkmark \\
    & $CI_G$ & \Checkmark & \Checkmark     & \Checkmark   & \Checkmark & \Checkmark & \Checkmark\\
    \midrule
    $\nu = 2.8$ & $CI_t^{\mathbf{seq}}$ & \Checkmark &  \Checkmark & \Checkmark & \Checkmark  & \Checkmark & \Checkmark\\
    & $CI_G$ & \Checkmark     & \Checkmark   & \Checkmark & \Checkmark & \Checkmark & \Checkmark\\
    \midrule
    $\nu = 4$  & $CI_t^{\mathbf{seq}}$ & \Checkmark &  \Checkmark & \Checkmark & \Checkmark  & \Checkmark & \Checkmark\\
    & $CI_G$ & \Checkmark     & \Checkmark   & \Checkmark & \Checkmark & \Checkmark & \Checkmark\\
    \bottomrule
  \end{tabular}
\end{table}


\section{ILLUSTRATION OF CONFIDENCE REGIONS}
We plot confidence regions for one realizations of Gaussian process with smoothness $\nu=1.5$. The iteration number is 30. The results are shown in Figure \ref{figCRegion}. It can be seen from Figure \ref{figCRegion} that our confidence region is more conservative than the naive one.

\begin{figure*}[h!]
  \centering
  \includegraphics[width=0.44\textwidth]{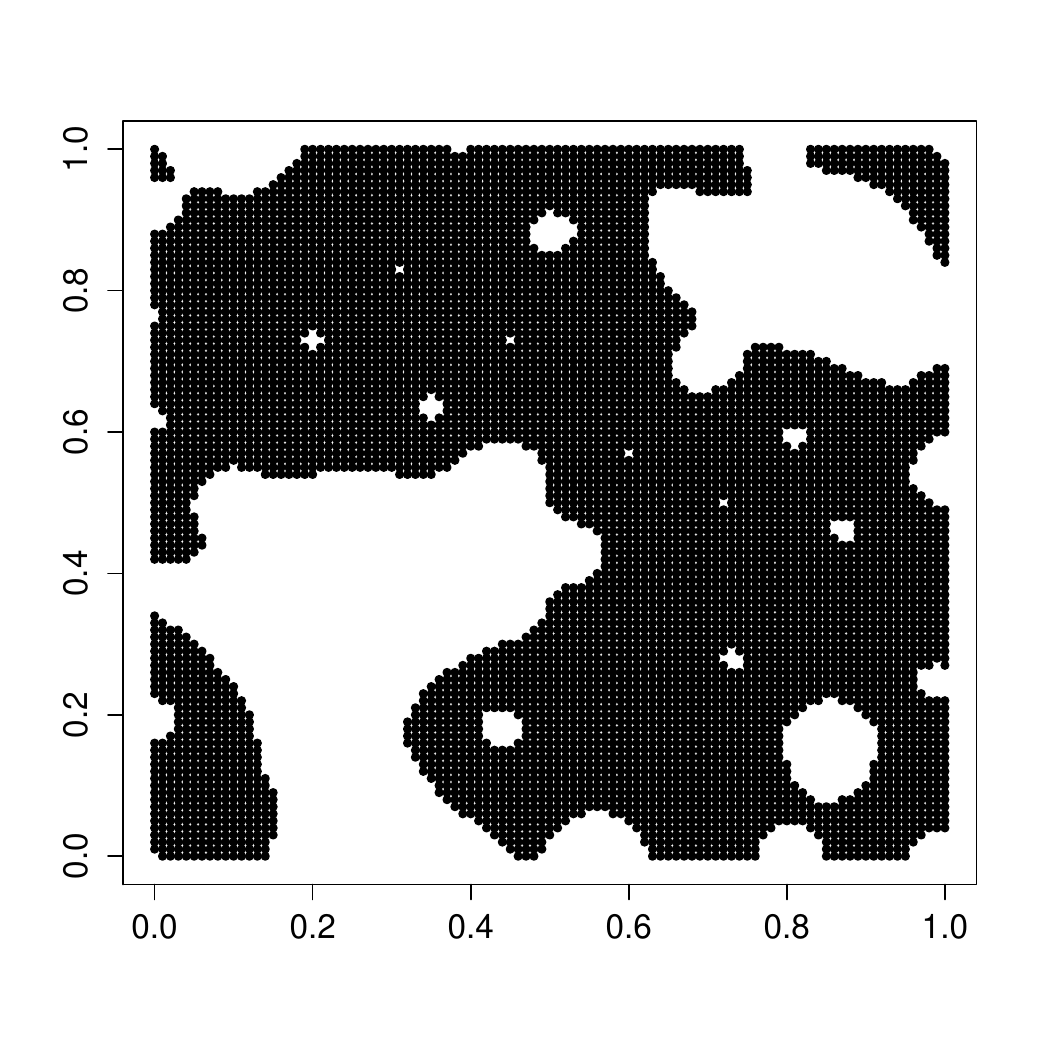}
        \includegraphics[width=0.44\textwidth]{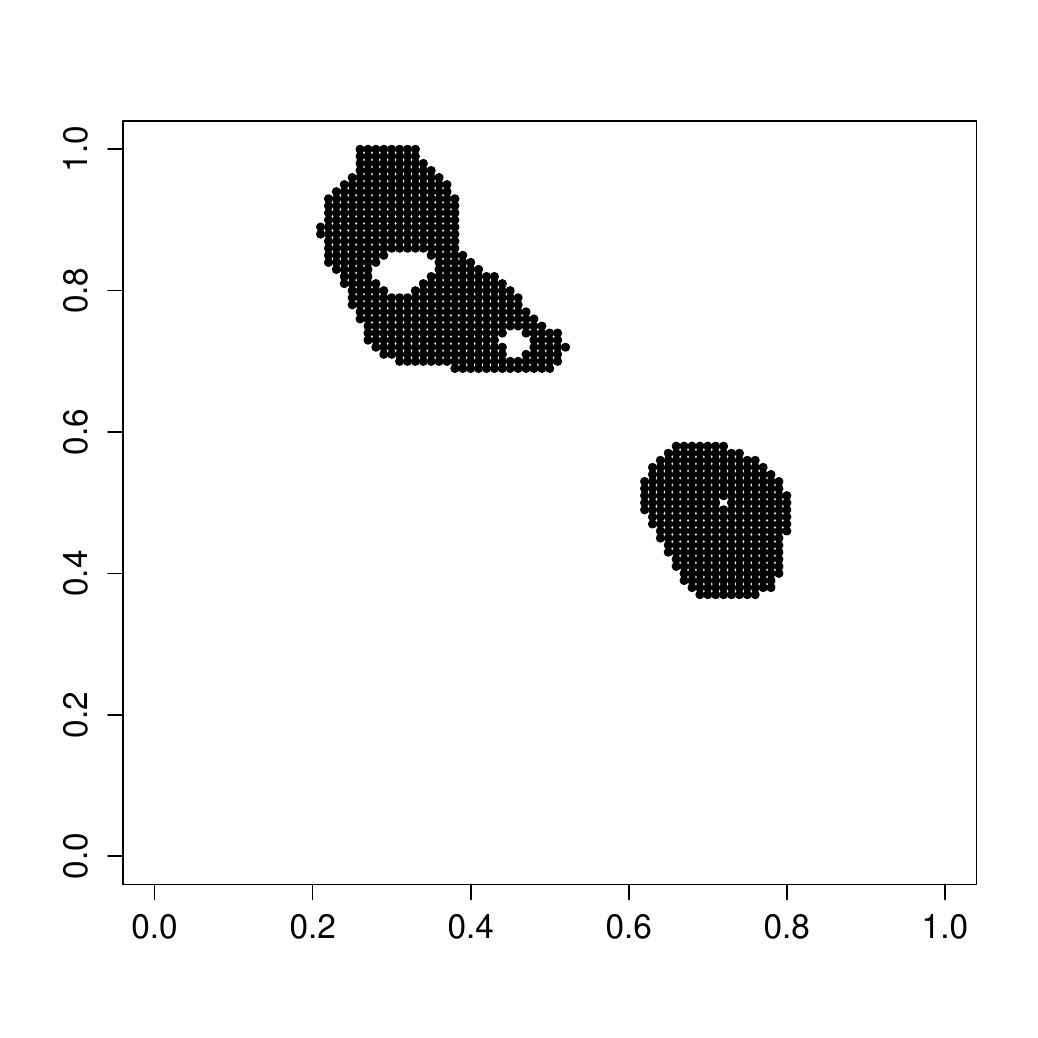}
  \caption{Confidence region of $CI_t^{\mathbf{seq}}$ (Panel 1) and $CI_G$ (Panel 2). The nominal confidence level is 95\%. }
  \label{figCRegion}
\end{figure*}


\vfill


\end{document}